\documentclass[11pt,a4paper,reqno]{amsart}
\usepackage{cite}

\usepackage{amsmath,amssymb,amsfonts,amsthm,epsfig,graphicx}
\usepackage{graphicx,color}
\usepackage{url}

\def\E{\mathcal{E}}

\def\H{\mathcal{H}}

\def\N{\mathbb N}

\def\R{\mathbb R}

\def\C{\mathbf{C}}

\def\HH{\mathbf{H}}

\def\L{\Lambda}
\def\Om{\Omega}
\def\S{\Sigma}

\def\a{\alpha}

\def\g{\gamma}
\def\de{\delta}
\def\e{\varepsilon}
\def\k{\kappa}
\def\l{\lambda}
\def\s{\sigma}

\def\om{\omega}
\def\vphi{\varphi}

\def\Lip{{\rm Lip}}

\def\Div{{\rm div}\,}
\def\Id{{\rm Id}\,}
\def\dist{{\rm dist}}
\def\loc{{\rm loc}}
\def\diam{{\rm diam}}

\def\spt{{\rm spt}}

\def\weak{\stackrel{*}{\rightharpoonup}}

\def\pa{\partial}

\def\cc{\subset\subset}

\def\HH{\mathbf{H}}

\def\C{\mathbf{C}}
\def\D{\mathbf{D}}
\def\K{\mathbf{K}}

\def\ov{\overline}
\theoremstyle{plain}
\newtheorem{theorem}{Theorem}[section]
\newtheorem{lemma}[theorem]{Lemma}
\newtheorem{corollary}[theorem]{Corollary}
\newtheorem{proposition}[theorem]{Proposition}
\newtheorem*{theorem*}{Theorem}
\newtheorem*{corollary*}{Corollary}

\theoremstyle{definition}
\newtheorem{definition}[theorem]{Definition}

\newtheorem{remark}[theorem]{Remark}

\newtheorem*{notation*}{Notation}

\numberwithin{equation}{section}
\numberwithin{figure}{section}

\renewcommand{\le}{\leqslant}

\renewcommand{\ge}{\geqslant}

\title[Capillarity problems with nonlocal surface tension energies]{Capillarity problems with \\ nonlocal surface tension energies}

\author{Francesco Maggi}
\address[Francesco Maggi]{Abdus Salam International Center for Theoretical Physics, Strada Costiera 11, I-34151, Trieste, Italy. On leave from the
University of Texas at Austin}
\email{fmaggi@ictp.it}

\author{Enrico Valdinoci}
\address[Enrico Valdinoci]{School of Mathematics and Statistics,
University of Melbourne,
813 Swanston Street, Parkville VIC 3010, Australia, and
Weierstra{\ss} Institut f\"ur Angewandte Analysis und Stochastik,
Mohrenstra{\ss}e 39, 10117 Berlin, Germany,
and Dipartimento di Matematica, Universit\`a degli studi di Milano,
Via Saldini 50, 20133 Milan, Italy}
\email{enrico@mat.uniroma3.it}

\begin{document}

\begin{abstract}
{\rm We explore the possibility of modifying the classical Gauss free energy functional used in capillarity theory
by considering surface tension energies of nonlocal type. The corresponding variational principles lead to new equilibrium
conditions which are compared to the mean curvature equation and Young's law found in classical capillarity
theory. As a special case of this family of problems we recover a nonlocal relative isoperimetric problem of geometric interest.}
\end{abstract}

\maketitle

\section{Introduction}
\subsection{Overview}
Classical capillarity theory is based on the study of volume-constrained critical points and local/global minimizers
of the Gauss free energy of a liquid droplet occupying a region $E$ inside a container $\Om\subset\R^n$, $n\ge 2$.
If $\H^{n-1}$ denotes the $(n-1)$-dimensional Hausdorff measure in $\R^n$, then the Gauss free energy of $E$ is
\begin{equation}
  \label{gauss free energy}
  \H^{n-1}(\Om\cap\pa E)+\s\,\H^{n-1}(\pa\Om\cap\pa E)+\int_E g(x)\,dx
\end{equation}
where $\H^{n-1}(\Om\cap\pa E)$ accounts for the surface tension energy of the interior liquid/air interface,
$\s\,\H^{n-1}(\pa\Om\cap\pa E)$ for the surface tension energy due to the liquid/solid interface
(measured relatively to the liquid/air tension, so that the {\it relative adhesion coefficient} $\s$ is assumed
to satisfy $-1<\s<1$), and where $g(x)$ stands for the potential energy density acting on the droplet.
It is well-known that when $E$ is a volume-constrained
critical point of the Gauss free energy having sufficiently smooth boundary, then the equilibrium conditions (Euler-Lagrange equations)
for $E$ take the form
\begin{eqnarray}\label{mean curvature equation}
  \HH_{\pa E}(x)+g(x)=c\,,&&\qquad\mbox{for every $x\in\Om\cap\pa E$}\,,
  \\\label{young law}
  \nu_E(x)\cdot\nu_\Om(x)=\s\,,&&\qquad\mbox{for every $x\in\overline{\Om\cap\pa E}\cap\pa\Om$}\,,
\end{eqnarray}
where $\nu_E$ is the outer unit normal to $E$, $\HH_{\pa E}$ is the mean curvature of $\pa E$ (computed with
respect to $\nu_E$) and $c\in\R$ is a Lagrange multiplier.

In this paper we introduce and investigate a family of capillarity-type energies where the effect of surface tension is measured
through nonlocal interaction energies, rather then through surface area. Given $s\in(0,1)$ and $\e\in(0,\infty]$ we denote by
\[
I_s^\e(E,F)=\int_E\,dx\int_F\frac{1_{(0,\e)}(|x-y|)\,dy}{|x-y|^{n+s}}
\]
the {\it fractional interaction energy of order $s$ truncated at scale $\e$} between two disjoint sets $E$ and $F$ contained in $\R^n$. We then work with the following ``fractional Gauss free energy''
\begin{equation}
  \label{nonlocal free energy}
  I_s^\e(E,\Om\cap E^c)+\s\,I_s^\e(E,\Om^c)+\int_E\,g(x)\,dx\,,
\end{equation}
Points in $E$ interact with points in $\Om\cap E^c$ and with points in $\Om^c$; the second type of interaction is weighted by a constant $\s$ having the same role of the relative adhesion coefficient in the classical model, and interactions are truncated at distance $\e$. Since the kernel $|z|^{-n-s}$ is not locally integrable, the function $x\in E\mapsto\int_{E^c}|x-y|^{-n-s}\,dy$ explodes like $\dist(x,\pa E)^{-s}$
as $x\in E$ approaches the boundary of $E$. Now for every $y\in\pa E$ the function $t>0\mapsto\dist(y-t\nu_E(y),\pa E)^{-s}=t^{-s}$ is integrable as $t\to 0^+$,
and thus we understand a term like the integral over $x\in E$ of $x\in E\mapsto\int_{E^c}|x-y|^{-n-s}\,dy$, or more generally $I_s^\e(E,E^c)$ with $\e<\infty$, as a nonlocal measurement of the surface area of $\pa E$. This intuition is confirmed by the fact that, in the limit $s\to 1^-$ corresponding to highly concentrated kernels, and after scaling by the factor $(1-s)$, the nonlocal capillarity energy \eqref{nonlocal free energy} converges to its local counterpart \eqref{gauss free energy},
\[
\lim_{s\to 1^-}\frac{(1-s)}{\k_n}\,\Big(  I_s^\e(E,\Om\cap E^c)+\s\,I_s^\e(E,\Om^c)\Big)=\H^{n-1}(\Om\cap\pa E)+\s\,\H^{n-1}(\pa\Om\cap\pa E)
\]
see Proposition \ref{prop classical} below. The latter property indicates that for $s$ close to $1$ the nonlocal model is quite close to the classical one. There are however some qualitative differences of possible interest, and the goal of this paper is starting their study.

Clearly, in order to understand these differences, the first step is deriving and discussing the Euler-Lagrange equations for the nonlocal capillarity energy \eqref{nonlocal free energy}. Both the interior equilibrium condition \eqref{mean curvature equation} and
Young's law (the contact angle condition \eqref{young law}) are affected by the non-locality of the model.

A first remarkable difference is that the interior equilibrium condition feels the effect of the
relative adhesion coefficient $\s$ at interior points whose distance from $\pa\Om$ is within the range
of the interaction kernel. (This is in striking difference with the classical model, where the corresponding interior equilibrium condition, namely
\eqref{mean curvature equation}, is completely unaffected by the mismatch in surface tension even at points in the
boundary of the droplet lying at arbitrarily small distance from the container walls.) Indeed, as proved in Theorem \ref{thm EL equation} below, the interior equilibrium condition in the fractional setting takes the form
\begin{equation}
  \label{mean curvature equation nonlocal}
  \HH^{s,\e}_{\pa E}(x)-(1-\s)\,\int_{\Om^c}\frac{1_{(0,\e)}(|x-y|)}{|x-y|^{n+s}}\,dy+g(x)=c\qquad\mbox{for every $x\in\Om\cap\pa E$}\,,
\end{equation}
where $\HH^{s,\e}_{\pa E}(x)$ is the {\it fractional mean curvature of $\pa E$ at $x$} (of fractional order $s$ and with truncation at scale $\e$), defined as
\[
\HH^{s,\e}_{\pa E}(x)={\rm p.v.}\int_{\R^n}\Big(1_{E^c}(y)-1_E(y)\Big)\,\frac{1_{(0,\e)}(|x-y|)}{|x-y|^{n+s}}\,dy\qquad\forall x\in \pa E\,.
\]
This last integral has to be defined in the principal value sense  and only for $x\in\pa E$, because in order for the integral to converge it is essential that, in a ball of radius $r>0$ centered at $x$, $1_{E^c}$  and $-1_E$ cancel out the presence of the non-integrable kernel on outside of a region of volume ${\rm o}(r^n)$. With this caveat in mind, it holds that, as $s\to 1^-$, $(1-s)\,\HH^{s,\e}_{\pa E}(x)\to \HH_{\pa E}(x)$ for every $x\in\pa E$ such that $\pa E$ is an hypersurface of class $C^2$ around $x$. The novel feature of the fractional model is contained in the second term on the left-hand side of \eqref{mean curvature equation nonlocal}, namely
\[
-(1-\s)\,\int_{\Om^c}\frac{1_{(0,\e)}(|x-y|)}{|x-y|^{n+s}}\,dy\,.
\]
Because of this term, the mismatch $1-\s$ in the surface tension between the liquid/air and liquid/solid interface is felt {\it also at point $x\in\Om\cap\pa E$ lying at a distance at most $\e$ from the boundary wall $\pa\Om$}. Notice that this nonlocal term, multiplied by $(1-s)$, converges to $0$ as $s\to 1^-$ for every $x\in\Om$.

Coming to the contact angle condition, as proved in Theorem \ref{thm nonlocal young} below, when working with the fractional model one finds a different contact angle than the one predicted in the classical Young's law \eqref{young law}. Independently from the considered value of $\e$ and on the ambient space dimension $n$, the fractional Young's law takes the form
\begin{equation}
\label{young law nonlocal intro}
  \nu_E(x)\cdot\nu_\Om(x)=\cos(\pi-\theta(s,\s))\,,\qquad\mbox{for every $x\in\overline{\Om\cap\pa E}\cap\pa\Om$}\,,
\end{equation}
where $\theta=\theta(s,\s)\in(0,\pi)$ is uniquely defined in terms of $s$ and $\s$ by the identity
\begin{eqnarray}\label{theta formula}
&&
\int_{\R^n}\frac{(1_{J_\theta^c\cap H}+
\s\,1_{H^c}-1_{J_\theta})(z)}{|e(\theta)-z|^{n+s}}\,dz\,=\,0,
\\ \nonumber
{\mbox{where }}&& J_\theta=\Big\{
x\in \R^n:\mbox{$x_n>0$ and $\cos\a\,x_n=\sin\a\,x_1$
for some $\a\in(0,\theta)$}\Big\},\\\nonumber
&& H=\{x\in\R^n:x_n>0\}
\\\nonumber
{\mbox{and }}
&&e(\theta)=\cos\,\theta\,e_1+\sin\theta\,e_n\,,
\end{eqnarray}
whose geometric significance is illustrated in
\begin{figure}
  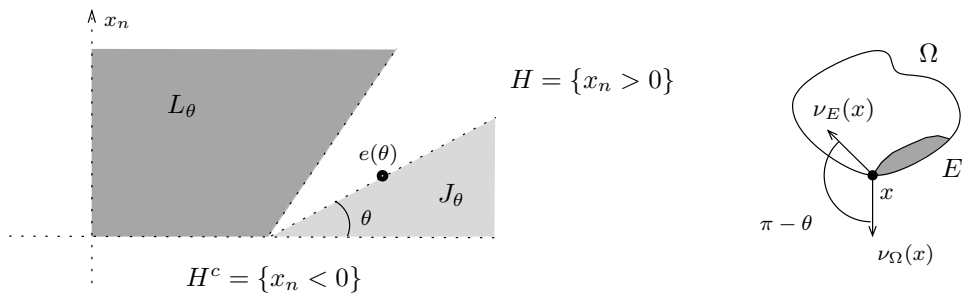\caption{\small{The contact angle for the fractional Young's law of order $s$ is computed by balancing the volume of the cone $L_\theta$ and the volume of $H^c$ multiplied by $\s$. In both cases ``volume'' is computed with respect to the singular density $|z-e(\theta)|^{-n-s}dz$, where both integral converge as the non-integrable singularity $e(\theta)$ is at positive distance from both $L_\theta$ and $H^c$. Notice that $L_\theta$ is defined by considering the reflection $J_\theta^*$ of $J_\theta$ with respect to $H\cap\pa J_\theta$, and then by setting $L_\theta=J_\theta^*\cap H$.}}\label{fig nonlocalyoungzero}
\end{figure}
Figure \ref{fig nonlocalyoungzero}. (Notice that the independence of $\theta$ from $n$ is not apparent from \eqref{theta formula}.) One has $\s\in(-1,1)\mapsto\theta(s,\s)$ is strictly increasing with
\[
\theta(s,0)=\frac\pi2\,,\quad\lim_{\s\to (-1)^+}\theta(s,\s)=0\,,\qquad\lim_{\s\to 1^-}\theta(s,\s)=\pi
\]
and, quite importantly,
\[
\lim_{s\to 1^-}\cos(\pi-\theta(s,\s))=\s\,,
\]
so that the fractional Young's law \eqref{young law nonlocal intro} converges to its local counterpart \eqref{young law} as $s\to 1^-$. The fact of obtaining a different contact angle than the classical one may be reconciliated with physical observation as the angle predicted by the classical Young's law may be actually observed in the nonlocal context {\it at a characteristic distance from the boundary of the container}. In other words, {\it the nonlocal model may predict different microscopic and macroscopic contact angles, the latter in accordance with \eqref{young law}}. We plan to address this issue in a subsequent paper, by focusing on the fractional sessile droplet problem.

Let us now comment on the mathematical background of our work. The use of fractional Sobolev norms in the analysis of partial differential
equations is of course a well established area of research with a vast literature and a huge range of applications.
The study of nonlocal {\it geometric} variational problems has attracted a large
attention since the seminal work \cite{caffaroquesavin}, where {\it nonlocal minimal surfaces} have been introduced motivated by the study
of the mean curvature flow as the limit of a process based on long range correlation. The boundary of a set $E$ is nonlocal area minimizing
in an open set $\Om$ if the quantity
\[
\int_{E\cap\Om}\,\int_{E^c\cap\Om}\frac{dxdy}{|x-y|^{n+s}}+\mbox{additional ``lower order'' interaction terms}
\]
is minimized by $E$ among all sets $F$ such that $F\setminus\Om=E\setminus\Om$. The main result in \cite{caffaroquesavin} is partial $C^{1,\a}$-regularity theorem outside a closed singular set of dimension $n-2$. Higher order regularity and improved dimensional estimates for the singular set have been obtained in \cite{savinvaldinoci,barrerafigallivaldinoci,figallivaldinociLipschitz}, examples of singular minimizing cones have been obtained in \cite{daviladelpinowei,daviladelpinowei2}, while boundaries with constant fractional mean curvature have been studied in \cite{cabrefallweth1,cabrefallweth2,ciraolofigallimagginovaga,daviladelpiniodipierrovaldinoci}. The present paper is also a contribution to the developing theory of nonlocal geometric variational problems. Indeed the minimization of \eqref{nonlocal free energy} in the case $\s=0$, $g=0$, and $\e=+\infty$ leads to study a family of {\it relative isoperimetric problems for fractional perimeters} in the open set $\Om$. Relative isoperimetric problems are of course a classical subject in the calculus of variations, especially because of their importance in determining (or in bounding) sharp constants in Poincar\'e-type inequalities; see \cite{mazyaBOOK}. This kind of application uses the possibility of writing Dirichlet energies as perimeter integrals over super-level sets by the coarea formula. This is possible also in the nonlocal case, where an appropriate version of the coarea formula can be found, for example, in \cite{visintin}.

\subsection{Interaction kernels} The study of nonlocal geometric variational problems is mainly concerned with the nonlocal perimeters
defined through the infinite-range isotropic singular kernels or, briefly, fractional kernels. Given $s\in(0,1)$, the {\it fractional kernel} of order $s$ is defined as
\begin{equation}
  \label{fractional kernel}
  K_s(\zeta)=\frac1{|\zeta|^{n+s}}\,,\qquad \zeta\in\R^n\setminus\{0\}\,.
\end{equation}
It also seems interesting to consider finite range interactions. We thus introduce the
{\it truncated fractional kernel} of order $s$,
\begin{equation}
  \label{truncated fractional kernel}
  K_s^\e(\zeta)=\frac{1_{(0,\e)}(|\zeta|)}{|\zeta|^{n+s}}\,,\qquad \zeta\in\R^n\setminus\{0\}\,,\e\in(0,\infty]\,.
\end{equation}
Given $K_s$ and $K_s^\e$ as the prototype kernels in our theory, we may finally want to consider possibly anisotropic interactions. We are thus led to introduce the following family of kernels.

Given~$n\ge 2$,~$s\in(0,1)$, $\l\ge1$ and $\e\in[0,\infty]$ we consider the family of {\it interaction kernels}
\[
\K(n,s,\l,\e)\qquad\mbox{(and set $\K(n,s,\l)=\K(n,s,\l,0)$)}
\]
consisting of those {\it even} functions~$K:\R^n\setminus\{0\}\to[0,+\infty)$ satisfying
\begin{equation}
  \label{kernelclass 0}
  \frac{1_{B_\e}(\zeta)}{\l\,|\zeta|^{n+s}}\le K(\zeta)\le\frac{\l}{|\zeta|^{n+s}}
  \qquad\forall\zeta\in\R^n\setminus\{0\}\,.
\end{equation}
(Here, $B_\e(x)$ is the ball
 of center $x$ and radius $\e$,
and we simply set $B_\e=B_\e(0)$.)
In particular, we assume that~$K$ is bounded from above by a homogeneous
kernel with polynomial decay of degree~$-(n+s)$
and that is bounded from below
by the same type of homogeneous
kernels up to distance $\e$ from the origin. Notice that $\K(n,s,1,\infty)$ contains only the
fractional kernel $K_s$ defined in \eqref{fractional kernel}. Given any $K\in\K(n,s,\l,\e)$ we set
\begin{eqnarray}\label{homogeneous kernel}
K^*(\zeta)=\lim_{r\to 0^+}r^{n+s}\,K(r\,\zeta)\qquad\zeta\ne 0\,,
\end{eqnarray}
provided the limit exists. Notice that $K^*$ is automatically $-(n+s)$-homogeneous and bounded from above by $\l\,|\zeta|^{-n-s}$, and that in the case
of truncated fractional kernels we have
\[
(K^\e_s)^*=K_s\qquad\forall s\in(0,1)\,,\,\e>0\,.
\]
Occasionally we shall need to work with smoother interaction kernels: given $h\in\N$ we thus introduce the class
\[
\K^h(n,s,\l,\e)\qquad\mbox{(and set $\K^h(n,s,\l)=\K^h(n,s,\l,0)$)}
\]
consisting of those $K\in\K(n,s,\l,\e)\cap C^h(\R^n\setminus\{0\})$, with
\begin{equation}\label{D2}
|D^j K(\zeta)|\le \frac{\l}{|\zeta|^{n+s+j}}\qquad \forall \zeta\in\R^n\setminus\{0\}\,,1\le j\le h\,.
\end{equation}
Each kernel $K$ defines an {\it interaction functional} between disjoint subsets of $\R^n$,
\[
I(E,F)=\int_E\int_FK(x-y)\,dx\,dy\in[0,\infty]\,,\qquad E,F\subset\R^n\,,E\cap F=\varnothing\,.
\]
The {\it nonlocal perimeter associated to $K$} is defined as the interaction of a set with its complement
\[
P(E)=I(E,E^c)\,,\qquad E^c=\R^n\setminus E\,.
\]
In the important cases of the  fractional kernel $K=K_s$ and of truncated fractional kernel $K=K_s^\e$ we write $I_s$ and $I_s^\e$ in place of $I$, and $P_s$ and $P_s^\e$ in place of $P$, so that
\[
I_s(E,F)=\int_E\int_F\frac{dx\,dy}{|x-y|^{n+s}}\,,\qquad P_s(E)= I_s(E,E^c)\,,
\]
\[
I_s^\e(E,F)=\int_E\int_F\frac{1_{B_\e}(x-y)\,dx\,dy}{|x-y|^{n+s}}\,,\qquad P_s^\e(E)= I_s^\e(E,E^c)\,.
\]
As shown in \cite{davilaBV} (see also \cite{Bourgain01anotherlook})
\[
\lim_{s\to 1^-}(1-s)\,P_s^\e(E)=\k_n\,\H^{n-1}(\pa^*E)\qquad   \k_n:=\frac12\int_{S^n}|e\cdot\om|d\H^{n-1}_\om\qquad e\in S^{n-1}\,,
\]
whenever $E$ is a set of finite perimeter in $\R^n$ and $\pa^*E$ denotes the reduced boundary of $E$ (for example, if $E$ is a bounded open set with Lipschitz boundary, then $E$ is a set of finite perimeter and $\pa^*E=\pa E$).

\subsection{Nonlocal capillarity energy} Given $K\in\K(n,s,\l,\e)$, an open set $\Om\subset\R^n$, and $\s\in(-1,1)$ we define the nonlocal capillarity energy of $E\subset\Om$ as
\begin{equation}\label{1.2bis}
\E(E)=I(E,E^c\Om)+\s\,I(E,\Om^c)\,.
\end{equation}
Here and in the following we adopt the following unusual convention
in order to simplify formulas involving  the interaction functional: precisely, when a set
intersection $F\cap G$ will appear as an argument of $I$, we shall write $FG$ in place
of $F\cap G$. For example,
\begin{equation}
  \label{convention}
  \mbox{$I(EF,GH)$ stands for $I(E \cap F,G\cap H)$}\,.
\end{equation}
Looking at \eqref{1.2bis}, the term $I(E,E^c\Om)$ accounts for interactions between liquid and air particles, while the term $I(E,\Om^c)$
 accounts for interactions between $E$ and the solid walls of the container. From the physical
 point of view, we expect short range interactions to matter the most. When working with the fractional kernels $I_s^\e$, this can be taken into account either by requiring the truncation parameter $\e$ to be small, or by taking $s$ close to $1$. As already noticed, the latter option corresponds to highly concentrated kernels whose fractional perimeter are increasingly close to the classical perimeter.

The basic variational problem we are interested in is then
\begin{equation}
  \label{variational problem}
  \g=\inf\Big\{\E(E)+\int_E\,g(x)\,dx:\;\, E\subset\Om\,,\;\,|E|=m\Big\}
\end{equation}
where $m\in(0,|\Om|)$ and $g:\R^n\to\R$ are given. As already noticed, when~$\sigma=0$ and~$g=0$,
\eqref{variational problem} is a nonlocal relative isoperimetric problem of geometric and functional interest.
The minimization problem in~\eqref{variational problem} is indeed well-posed, according to the following simple
result:

\begin{proposition}[Existence of minimizers]\label{lemma esistenza}
If $K\in\K(n,s,\l)$, $\Omega$ is an open bounded set with $P(\Om)<\infty$,
 and~$g\in L^\infty(\Om)$, then there exist minimizers in \eqref{variational problem}.
 Moreover, $I(E,E^c\Om)<\infty$ for every minimizer $E$.
\end{proposition}

We have already mentioned the fact that, as $s\to 1^-$, fractional perimeters converge to classical perimeters.
This is true also for our nonlocal capillarity energy.

\begin{proposition}
  [Convergence to the classical energy]\label{prop classical}
  If $\Om$ and $E$ are open sets with Lipschitz boundary and $E\subset\Om$, then
  \begin{eqnarray*}
  \lim_{s\to 1^-}(1-s)\,I_s(E,E^c\Om)&=&\k_n\,\H^{n-1}(\Om\cap\pa E)
  \\
  \lim_{s\to 1^-}(1-s)\,I_s(E,\Om^c)&=&\k_n\,\H^{n-1}(\pa E\cap\pa\Om)\,.
  \end{eqnarray*}
  In particular,
  \[
  \lim_{s\to 1^-}\frac{(1-s)}{\k_n}\,\E(E)=\H^{n-1}(\Om\cap\pa E)+\s\,\H^{n-1}(\pa\Om\cap\pa E)\,.
  \]
\end{proposition}

\subsection{Euler-Lagrange equations} We now address the form
taken by the equilibrium conditions (Euler-Lagrange equations)
at boundary points of minimizers in the nonlocal capillarity problem. Notice that a minimizer $E$ in \eqref{variational problem}
could be in principle quite irregular, and actually
the property of being a minimizer is invariant under modifications of $E$ on and by a set of volume zero.
It is thus convenient to work with a robust notion of boundary of $E$ and set
\[
\pa E=\Big\{x\in\overline{\Om}:0<|E\cap B_r(x)|<\om_n\,r^n\ \ \forall r>0\Big\}\,.
\]
We shall then define the {\it regular part ${\rm Reg}_E$} and the {\it singular part $\S_E$} of $\pa E$ by
setting
\[
{\rm Reg}_E=\left\{x\in\overline{\Om\cap\pa E}:
\begin{split}
&\mbox{there exists ${{\varrho}}>0$ and $\a\in(s,1)$ s.t. $B_{{\varrho}}(x)\cap\pa E$ is a $C^{1,\a}$-manifold}
\\
&\mbox{with boundary, whose boundary points are in $\pa\Om$}
\end{split}\right\}
\]
and $\S_E=\pa E\setminus{\rm Reg}_E$, respectively. We expect the Euler-Lagrange equations
to hold in weak form at every point $x\in\pa E$ and in a stronger, pointwise form at every $x\in{\rm Reg}_E$;
see \eqref{weak euler lagrange intro} and \eqref{stationary set} below.
Since our primary goal here is understanding the qualitative features of the proposed
nonlocal capillarity model, and thus its possible physical interest, we shall not be concerned with the regularity
problem, which would consists in showing the smallness of $\S_E$. Let us recall that, in the local case, when $n=3$ the singular set is empty \cite{taylor77,luckhaus,dephilippismaggiARMA}.

In order to introduce the Euler-Lagrange equations for the nonlocal capillarity energy $\E$, it is convenient
to recall the form taken by the equilibrium conditions for local minimizers of nonlocal perimeters.
Given two sets $E$ and $F$ which
are equal outside of a bounded open set $A$ we formally have
\[
P(E)-P(F)=P(E,A)-P(F,A)
\]
where we have set
\[
P(E;A)=I(EA,E^cA)+I(EA,E^cA^c)+I(E^cA,EA^c)\,,
\]
and where the identity $P(E)-P(F)$ holds in general only in a formal sense as it involves the cancellation of the possibly infinite
interaction terms $I(EA^c,E^cA^c)=I(FA^c,F^cA^c)$ (as $E\cap A^c=F\cap A^c$ by assumption). We thus say that $E\subset\R^n$
is a critical point of $P$ in a bounded open set $A$ if
\[
\frac{d}{dt}\bigg|_{t=0}P(f_t(E),A)=0\,,
\]
for every family of diffeomorphisms $\{f_t\}_{|t|<\de}$ such that
\begin{equation}
  \label{diffeos critical point no constraint}
  f_0=\Id\qquad
  \spt(f_t-\Id)\cc A\qquad\forall |t|<\de\,.
\end{equation}
If $K\in \K^1(n,s,\l)$, then being a critical point is equivalent to the condition
\begin{equation}
  \label{weak euler lagrange P}
  \int_E\,\int_{E^c}\,\Div_{(x,y)}\big(K(x-y)\,(T(x),T(y))\big)\,dx\,dy=0\qquad\forall T\in C^1_c(A;\R^n)\,.
\end{equation}
where we have set
\[
\Div_{(x,y)}\big(K(x-y)\,(T(x),T(y))\big)=\Div_x\big(K(x-y)\,T(x)\big)+\Div_y\big(K(x-y)\,T(y)\big)\,.
\]
We refer to \eqref{weak euler lagrange P} as to the {\it weak form} of the Euler-Lagrange equation of $P$ in $A$. Notice that
\eqref{weak euler lagrange P} ``holds at every $x\in\pa E$'' in the sense that it is satisfied by every measurable set $E$ if restricted to vector fields $T$ with $\spt\,T\cap\pa E=\varnothing$. If $K\in \K^2(n,s,\l)$, then \eqref{weak euler lagrange P} implies that
\begin{equation}
  \label{strong euler lagrange P}
  \HH^K_{\pa E}(x)=0\qquad\forall x\in A\cap{\rm Reg_E}
\end{equation}
where $\HH^K_{\pa E}(x)$ is the {\it nonlocal mean curvature of $\pa E$ at $x$} (with respect to the kernel $K$), and is defined as
\begin{equation}
    \label{nonlocal mean curvature}
      \HH^K_{\pa E}(x):={\rm p.v.}\int_{\R^n}\Big(1_{E^c}(y)-1_E(y)\Big)\,K(x-y)\,dy\qquad x\in\pa E\,.
\end{equation}
This integral converges in the principal value sense as soon as $E$ is the epigraph of a $C^{1,\a}$-function with $\a>s$ in a neighborhood of $x$, and actually $\HH^K_{\pa E}$ is a continuous function on ${\rm Reg}_E$. Equation \eqref{strong euler lagrange P} is the {\it strong form} of \eqref{weak euler lagrange P}, and in the limit $s\to 1^-$ of
highly concentrated fractional kernels we have
\[
\lim_{s\to 1^-}(1-s)\,\HH^{K^\e_s}_{\pa E}(x)=\HH_{\pa E}(x)
\]
provided $\pa E$ is of class $C^2$ in a neighborhood of $x$.

Coming back to the capillarity problem, we say that $E\subset\Om$ is a (volume-constrained)
{\it critical point of $\E+\int g$} if
\begin{equation}
  \label{critical point}
  \frac{d}{dt}\bigg|_{t=0}\E(f_t(E))+\int_{f_t(E)}g=0\,,
\end{equation}
for every family of diffeomorphisms $\{f_t\}_{|t|<\de}$ such that, for every $|t|<\de$,
\begin{equation}
  \label{diffeos critical point}
  f_0=\Id\,,\quad
  \spt(f_t-\Id)\cc\R^n\,,\quad
  f_t(\Om)=\Om\,,\quad
  |f_t(E)|=|E|\,.
\end{equation}
Global minimizers
in \eqref{variational problem} are of course critical sets.  At regular points of a critical set of $\E+\int g$ the Euler-Lagrange equations
take the following form.

\begin{theorem}[Euler-Lagrange equation]\label{thm EL equation}
  Let $\Om$ be a bounded open set with $C^1$-boundary, $g\in C^1(\R^n)$, and $E$ be a critical point of $\E+\int g$. If $K\in\K^1(n,s,\s)$, then
  there exists a constant $c\in\R$ such that
  \begin{equation}
  \label{weak euler lagrange intro}
  \begin{split}
  &\iint_{E\times(E^c\cap\Om)}\Div_{(x,y)}\big(K(x-y)\,(T(x),T(y))\big)\,dxdy
  \\
  &+\s\iint_{E\times\Om^c}
  \Div_{(x,y)}\big(K(x-y)\,(T(x),T(y))\big)\,dxdy+\int_E\,\Div(g\,T)=c\,\int_E\,\Div\,T
  \end{split}
  \end{equation}
  for every $T\in C^\infty_c(\R^n;\R^n)$ with
  \[
  T\cdot\nu_\Om=0\qquad\mbox{on $\pa\Om$}\,.
  \]
  Moreover, if $K\in\K^2(n,s,\s)$, then
  \begin{equation}
    \label{stationary set}
      \HH^K_{\pa E}(x)-(1-\s)\,\int_{\Om^c}K(x-y)\,dy+g(x)=c\,,\qquad\forall x\in\Om\cap{\rm Reg}_E\,.
  \end{equation}
\end{theorem}

We next investigate the contact angle condition, or Young's law, in the nonlocal setting. Let us recall that in the local setting
Young's law can be derived through integration by parts starting from the weak form of \eqref{mean curvature equation}, that is
\[
\int_{\pa E}\Div^{\pa E}T+\int_{\pa E}g\,(T\cdot\nu_E)=c\,\int_{\pa E}T\cdot\nu_E
\]
for every $T\in C^1_c(\R^n;\R^n)$ with $T\cdot\nu_\Om=0$ on $\pa \Om$; see, e.g., \cite[Theorem 19.8]{maggiBOOK}, and compare with \eqref{weak euler lagrange intro}.
In the nonlocal case we need to use a different approach, avoiding integration by parts. More precisely, the nonlocal
Young's law will be obtained
by taking blow-ups of \eqref{stationary set} along sequences of regular interior points converging to $\pa\Om\cap{\rm Reg}_E$. Here and in the following we shall use the notation
\[
A^{x_0,r}=\frac{A-x_0}r
\]
for the blow-up of $A\subset\R^n$ at scale $r>0$ around $x_0\in\R^n$.

\begin{theorem}[Nonlocal Young's law]\label{thm nonlocal young}
  Let $K\in\K^2(n,s,\l)$ be such that the homogeneous kernel $K^*$ is well-defined accordingly to \eqref{homogeneous kernel}, and let $g\in C^0(\R^n)$. Let $\Om$ be a bounded open set with $C^1$-boundary and $E$ be a volume-constrained
  critical set of $\E+\int g$. Given $x_0\in{\rm Reg}_E\cap\pa\Om$, let $H$ and $V$ be the half-spaces such that
  \[
  \mbox{$\Om^{x_0,r}\to H$ and $E^{x_0,r}\to H\cap V$ in $L^1_{{\rm loc}}(\R^n)$ as $r\to 0^+$}
  \]
  and set $\nu_E(x_0):=\nu_V(0)$. Then the angle between $H$ and $V$ must satisfy the identity
  \begin{equation}
    \label{young law nonlocal general}
    \HH^{K^*}_{\pa (H\cap V)}(v)-(1-\s)\,\int_{H^c}\,K^*(v-z)\,dz=0\,,\qquad\forall v\in H\cap\pa V\,,
  \end{equation}
  see
  \begin{figure}
    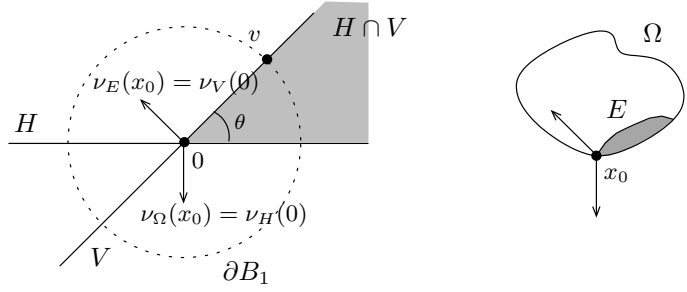\caption{{\small The nonlocal Young's law is computed at points $x_0\in\pa\Om$ where $E$ blow-ups a cone of the form $V\cap H$ where $V$ is an half-space, and $H$ is the half-space blow-up of $\Om$ at $x_0$. This law determines the angle between $V$ and $H$ via the identity \eqref{young law nonlocal general}.}}\label{fig nonlocalyoung}
  \end{figure}
  Figure \ref{fig nonlocalyoung}. In the special case when $K=K^\e_s$, and thus $K^*=K_s$, \eqref{young law nonlocal general} uniquely identifies the angle between $H$ and $V$. More precisely, for every $s\in(0,1)$ and $\s\in(-1,1)$ there exists a unique $\theta=\theta(s,\s)\in(0,\pi)$ such that
  \begin{equation}
  \label{young law fractional}
  \nu_E(x_0)\cdot\nu_\Om(x_0)=\nu_V(0)\cdot\nu_H(0)=\cos\big(\pi-\theta(s,\s)\big)\,.
  \end{equation}
  The function $\s\in(-1,1)\mapsto\theta(s,\s)$ is strictly increasing with
  \[
  \theta(s,0)=\frac\pi2\,,\quad\lim_{\s\to (-1)^+}\theta(s,\s)=0\,,\qquad\lim_{\s\to 1^-}\theta(s,\s)=\pi
  \]
  and
  \[
  \lim_{s\to 1^-}\cos\big(\pi-\theta(s,\s)\big)=\s\,.
  \]
  In particular, the fractional Young's law \eqref{young law fractional} converges to the classical Young's law in the limit $s\to 1^-$ of
  highly concentrated interaction kernels.
\end{theorem}

Theorem \ref{thm nonlocal young} shows that the nonlocal Young's law may take different forms depending on the considered kernels. Even in the class of isotropic fractional kernels $K_s$, the contact angle will depend on $s$ (in addition to its dependency on $\s$), although it will converge to the angle predicted by the classical Young's law in the limit $s\to 1^-$. The contact angle predicted by the classical Young's law may be actually observed in the nonlocal context at a characteristic distance from the boundary of the container. We plan to further investigate this issue in a subsequent paper, focusing on the sessile droplet problem.

We also remark that in the case $\s=0$ with isotropic kernel $K=K^\e_s$, the nonlocal Young's law always boils down to
\[
\nu_E(x_0)\cdot\nu_\Om(x_0)=0\qquad\forall x_0\in\pa\Om\cap{\rm Reg}_E\,.
\]
This is interesting as the corresponding variational problem
\[
\inf\Big\{I^\e_s(E,E^c\Om):E\subset\Om\,,|E|=m\Big\}
\]
is a natural fractional variant of the classical {\it relative isoperimetric problem in $\Om$}. Thus critical points in the relative isoperimetric problem and in all of its fractional variants share the same orthogonality condition at the boundary of $\Om$, independently from $\e$ and $s$. At the same time, the equilibrium interior condition $\HH^{K^\e_s}_{\pa E}={\rm constant}$ valid on $\Om\cap{\rm Reg}_E$ depends on the specific values of $s$ and $\e$.

\subsection{Interior regularity and other regularity properties} In the last part of our paper we address some regularity properties
of local (almost) minimizers of the nonlocal capillarity energy $\E$. In order to introduce the minimality condition that we shall consider,
let us notice that if $E$ and $F$ are equal outside of an open set $A$ (not necessarily contained in $\Om$,
\begin{figure}
  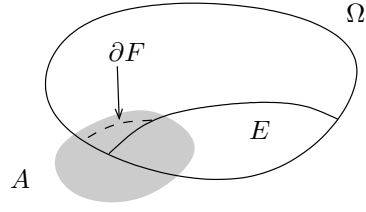\caption{{\small Since in Definition \ref{DFAL} we consider variations in an open set $A$ which is not contained in $\Omega$, we are in practice imposing on $E$ a Dirichlet condition along $\Omega\cap\pa A$ and a Neumann condition along $A\cap\pa\Om$.}}\label{fig almostmin}
\end{figure}
 see Figure \ref{fig almostmin}), that is, if  $F\cap A^c=E\cap A^c$, then one can formally compute (with the convention \eqref{convention} in force)
\begin{eqnarray*}
\E(E)-\E(F)
&=&I(E,E^c\Om)+\s\,I(E,\Om^c)-I(F,F^c\Om)-\s\,I(F,\Om^c)
\\
&=&I(E A,E^c\Om)+I(E A^c,E^c\Om A)+\s\,I(E A,\Om^c)
\\
&&-I(F A,F^c\Om)-I(F A^c,F^c \Om A)-\s\,I(F A,\Om^c)\,.
\end{eqnarray*}
We are thus led to consider the following kind of local (almost) minimality inequality.

\begin{definition}[Almost minimizers]\label{DFAL}
  {\rm Let $K\in\K(n,s,\l)$, $\Om$ and $A$ be open (possibly unbounded) sets in $\R^n$ such that
  \begin{equation}
    \label{condition Omega A}
      I(\Om A,\Om^c)<\infty\,,
  \end{equation}
  and let $\Lambda\in[0,\infty)$, $r_0\in(0,\infty]$ and $\s\in(-1,1)$. Given $E\subset\Om$, one says that $E$ is a {\it $(\Lambda,r_0,\s,K)$-minimizer in $(A,\Omega)$} if
  \begin{eqnarray}
    \label{Lambda r0 minimizer}
  &&I(E A,E^c \Om)+I(E A^c,E^c\Om A)+
\s\,I(E A,\Om^c)
  \\\nonumber
  &\le&
  I(F A,F^c \Om)+
I(F A^c,F^c \Om  A)+\s\,I(F A,\Om^c)+\Lambda\,|E\Delta F|\,,
  \end{eqnarray}
  for every $F\subset\Om$ with
$\diam(F\Delta E)<2\,r_0$ and $F\cap A^c=E\cap A^c$. Notice that \eqref{condition Omega A} guarantees that $I(F A,\Om^c)<\infty$ whenever $F\subset\Om$, so that, even when $\s<0$, the quantity
  \[
  I(F A,F^c\Om)+I(F A^c,F^c\Om A)+\s\,I(F A,\Om^c)\,,
  \]
  appearing on the right-hand side of \eqref{Lambda r0 minimizer} is well-defined in $(-\infty,\infty]$.}
\end{definition}

As proved in Corollary \ref{corollary minimi sono lambda minimi} below, if $E$ is a minimizer in \eqref{variational problem}, then there exist $\Lambda\ge0$ and $r_0>0$ (depending on $E$ and $\|g\|_{L^\infty(\Om)}$) such that $E$ is a $(\Lambda,r_0,\s,K)$-minimizer in $(\R^n,\Omega)$. The same is true for local minimizers of course, and the lower order term $\Lambda\,|E\Delta F|$ in the minimality inequality \eqref{Lambda r0 minimizer} actually allows to reabsorb various type of constraints (see \cite{Almgren76,tamanini} for more examples of this idea).

We are thus interested in understanding the regularity of $(\Lambda,r_0,\s,K)$-minimizers. Since at present an interior regularity theory for nonlocal variational problems has only been developed in the isotropic case of the fractional kernel $K_s$ (see \cite{caffaroquesavin,caputoguillen}) we shall mainly focus on this case. The first important remark is that on variations supported away from the boundary of $\Omega$, the minimality inequality \eqref{Lambda r0 minimizer} implies the type of almost--minimality
condition considered in \cite{caffaroquesavin,caputoguillen}. Thus, interior regularity is readily established.

\begin{theorem}[Interior regularity]\label{thm interior regularity}
  If $E$ is a $(\Lambda,r_0,\s,K_s)$-minimizer in $(A,\Om)$, then $A\cap\Om\cap{\rm Reg}_E$ is a $C^{1,\a}$-hypersurface for some universal $\a\in(0,1)$ and $A\cap\Om\cap\S_E$ is a closed set with Hausdorff dimension less than $n-3$.
\end{theorem}

The regularity problem near points on $\pa\Om$ is more complex than its interior counterpart because it involves the study of a free boundary. Here we just address what is usually the first step in the analysis of a regularity problem, namely, we obtain perimeter and volume density estimates which hold {\it uniformly} up to the boundary of $\Om$. This problem, in the case $\s<0$, presents some additional difficulties with respect to the interior case. These difficulties are addressed by exploiting some geometric inequalities for fractional perimeters.

\begin{theorem}
  [Density estimates]\label{thm density estimate intro}
  Let $n\ge 2$, $s\in(0,1)$, $\s\in(-1,1)$, $\Lambda\ge0$, and $K=K_s^\e$ for some $\e>0$.
  If $\Om$ is either a bounded open set with $C^1$ boundary or an half-space, then there exist positive constants $C_0$ (depending on $n$, $s$, $\s$, and $\Lambda$), $c_*$ (depending on $n$ and $s$) and $\k$ (depending on $n$, $s$, $\s$ and $\Om$) such that if $E$ is a $(\Lambda,r_0,\s,K_s^\e)$-minimizer in $(A,\Omega)$, then
  \begin{equation}
    \label{upper perimeter estimate intro}
      I_s^\e(E B_r(x),(E B_r(x))^c)\le C_0\,r^{n-s}\,,
  \end{equation}
  whenever $B_r(x)\subset A$ and $r<\min\{r_0,c_*\,\k,c_*\,\e\}$. Moreover,
  \begin{equation}
    \label{volume density estimates intro}
      \frac{1}{C_0}\le \frac{|E\cap B_r(x)|}{r^n}\le 1-\frac1{C_0}
  \end{equation}
  whenever $B_r(x)\subset A$, $r<\min\{r_0,c_*\,\k,c_*\,\e\}$, and $x\in\overline{\Om\cap\pa E}$.
\end{theorem}

\begin{remark}
  {\rm Theorem \ref{thm density estimate intro} holds for a much larger class of ``uniformly-$C^1$'' open sets, of which bounded open set with $C^1$-boundary and half-spaces are particular cases. The dependence of $\k$ from  $\Om$ can actually be expressed quite precisely in terms of this uniform $C^1$-property as explained in the course of the proof of Theorem \ref{thm density estimate intro}.}
\end{remark}


\subsection{Organization of the paper} In section \ref{section existence} we address the existence of minimizers in the nonlocal capillarity problem, and the convergence of the fractional capillarity energy to the classical Gauss free-energy in the limit $s\to 1^-$. In section \ref{section Euler lagrange} and section \ref{section nonlocal young} we discuss, respectively, the deduction of the Euler-Lagrange equations in weak and in strong form, and of the nonlocal Young's law. In section \ref{section interior regularity} we explain how to quickly deduce interior regularity, while section \ref{section boundary regularity} is devoted to the proof of Theorem \ref{thm density estimate intro}. Finally, in appendix \ref{appendix blowup limits} we obtain a quite natural closure result for sequences of almost-minimizers which shall be useful in future investigations.

\bigskip

\noindent {\bf Acknowledgment:} FM was supported by NSF-DMS Grant 1265910 and NSF-DMS FRG Grant 1361122. EV was supported by ERC grant 277749 E.P.S.I.L.O.N. Elliptic PDEs and Symmetry of Interfaces and Layers for Odd Nonlinearities and PRIN grant 201274FYK7 Critical Point Theory and Perturbative Methods for Nonlinear Differential Equations.




\section{Existence of minimizers and convergence to the classical energy}\label{section existence} We start by proving the existence of minimizers in the variational problem \eqref{variational problem}, namely
\begin{equation}
  \label{variational problem recall}
   \g=\inf\Big\{\E(E)+\int_E\,g(x)\,dx:\;\, E\subset\Om\,,\;\,|E|=m\Big\}
\end{equation}
under the assumptions that $K\in\K(n,s,\l,\e)$, $\Omega$ is an open bounded set with $P(\Om)<\infty$, and~$g\in L^\infty(\Om)$, and where
\[
\E(E)=I(E,E^c\Om)+\s\,I(E,\Om^c)\,;
\]
see Proposition \ref{lemma esistenza}. The proof is based on a semicontinuity argument and on a direct
minimization procedure. We premise the following lower semicontinuity lemma.

\begin{lemma}[Lower semicontinuity]\label{lemma lsc} If $P(\Om)<\infty$, $E_j\subset\Om$, and $E_j\to E$ in $L^1(\Om)$, then
  \[
  \liminf_{j\to\infty}\E(E_j)\ge \E(E)\,.
  \]
\end{lemma}

\begin{proof}
  This is immediate by Fatou's lemma if $\s\ge0$. If $\s\in(-1,0)$, then we exploit the identity
  \begin{eqnarray*}
  \E(E)&=&-P(\Om)+I(E,E^c\Om)+P(\Om)-|\s|\,I(E,\Om^c)
  \\
  &=&  -P(\Om)+I(E,E^c\Om)+(1-|\s|)\,I(E,\Om^c)+I(E^c\Om,\Om^c)\,,
  \end{eqnarray*}
  and, again, Fatou's lemma, to complete the proof.
\end{proof}

\begin{proof}[Proof of Proposition~\ref{lemma esistenza}] We first remark that since $K\in\K(n,s,\l,\e)$, then
and any~$p\in\R^n$,
\begin{equation}\label{VGAL}
P(F)\ge \frac1\l\,I_s(F\,B_{\e/2}(p),F^c B_{\e/2}(p))\,,\qquad\forall F\subset\R^n\,.
\end{equation}
Indeed, if~$x$, $y\in B_{\e/2}(p)$,
then~$|x-y|\le |x-p|+|p-y|<\e$ and so,
by~\eqref{kernelclass 0},
$$ P(F)=I(F,F^c)
\ge \int_{F\cap B_{\e/2}(p)}
\int_{F^c \cap B_{\e/2}(p)}
K(x-y)\,dx\,dy
\ge \frac1\l\,\int_{F\cap B_{\e/2}(p)}
\int_{F^c\cap B_{\e/2}(p)}
\frac{ dx\,dy}{|x-y|^{n+s}} ,$$
that proves~\eqref{VGAL}. Now, if~$H$ is a half-space such that $|H\cap\Om|=m$ and $R>0$ is such that
$\Om\subset B_R$, then
  \[
  \E(H\cap\Om)=I(H\Om,(H\Om)^c\Om)+\s\,I(H\Om,\Om^c)\le I(H\, B_R,H^c\, B_R)+P(\Om)<\infty\,,
  \]
since $I(H\,B_R,H^c\,B_R)\le C(n,s)\,R^{n-s}$
thanks to~\eqref{kernelclass 0}. As a consequence,
we find that $\g<\infty$. Let $E_j\subset\Om$ be such that $\E(E_j)+\int_{E_j}g\to \g$, then for $j$ large enough
  \[
  \g+1+\int_\Om|g|\ge I(E_j,E_j^c\Om)+\s\,I(E_j,\Om^c)\ge I(E_j,E_j^c\Om)-P(\Om)\,,
  \]
  and thus
  \begin{eqnarray*}
  P(E_j)=I(E_j,E_j^c\Om)+I(E_j,E_j^c\Om^c)
  \le\g+1+\int_\Om|g|+2\,P(\Om)\,.
  \end{eqnarray*}
  Since $E_j\subset B_R$, using this and~\eqref{VGAL}, we find that, up to extracting subsequences, $E_j\to E$ in $L^1_{{\rm loc}}(\R^n)$ for some $E\subset\Om$ with $|E|=m$. By Lemma \ref{lemma lsc}, we conclude that~$E$ is a minimizer. Now we remark that
  \begin{equation}\label{BO:Q1}
  I(E,\Om^c)\le I(\Om,\Om^c)=P(\Om)<+\infty,\end{equation}
  and so the fact that~$\E(E)<\infty$ also implies that
  \begin{equation}\label{BO:Q2}
  I(E,E^c\Om)<+\infty\,,
  \end{equation}
  as claimed.
\end{proof}

We now turn to the convergence of the fractional capillarity energy to the Gauss free energy in the limit $s\to 1^-$, that is, we prove Proposition \ref{prop classical}. Recalling that, by definition,
\[
\k_n=\int_{S^n}|e\cdot\om|d\H^{n-1}_\om
\]
we shall actually prove a stronger result, valid for every set of finite perimeter contained in $\Om$. Here $\pa^*E$ denotes the reduced boundary of the set of finite perimeter $E$, see \cite{maggiBOOK}.

\begin{proposition}\label{prop davila}
  If $\Om$ is an open set with Lipschitz boundary and $E\subset\Om$ is a set of finite perimeter with $I_s(E,E^c)<\infty$, then
  \begin{eqnarray}\label{davila tesi 1}
  \lim_{s\to 1^-}(1-s)\,I_s(E,E^c\Om)&=&\frac{\k_n}2\,\H^{n-1}(\Om\cap\pa^* E)
  \\\label{davila tesi 2}
  \lim_{s\to 1^-}(1-s)\,I_s(E,\Om^c)&=&\frac{\k_n}2\,\H^{n-1}(\pa^* E\cap\pa\Om)
  \end{eqnarray}
\end{proposition}

\begin{proof}
  Given $V\subset\R^n$ we define a Radon measure $\mu_s^V$ on $\R^n$ by setting
  \begin{eqnarray*}
  \mu_s^V(A)&=&(1-s)\,\Big(I_s(EA,E^cV)+I_s(EV,E^cA)\Big)
  \\
  &=&(1-s)\int_A\,dx\int_V|1_E(x)-1_E(y)|\,K_s(x-y)\,dy\,.
  \end{eqnarray*}
  Notice that $\mu_s^V(\R^n)$ is finite as $\mu_s^V(\R^n)\le 2\,(1-s)I_s(E,E^c)$. By \cite[Lemma 2]{davilaBV} we have that
  \begin{equation}
    \label{davila 1}
      \mu_s^V\weak\k_n\H^{n-1}\llcorner(V\cap\pa^*E)\quad\mbox{weakly-* as Radon measures in $V$}
  \end{equation}
  whenever $V$ is an open set, with
  \begin{equation}
    \label{davila 2}
      \k_n\,\H^{n-1}(V\cap\pa^*E)=\lim_{s\to 1^-}\mu_s^V(V)= 2\,\lim_{s\to 1^-}(1-s)\,I_s(EV,E^cV)
  \end{equation}
  provided $V$ is open, bounded, with Lipschitz boundary. By applying \eqref{davila 2} with $V=\Om$ we find that
  \[
  \k_n\,\H^{n-1}(\Om\cap\pa^*E)=2\,\lim_{s\to 1^-}(1-s)\,I_s(E,E^c\Om)\,,
  \]
  that is \eqref{davila tesi 1}. We now set
  \[
  N_r(A)=\Big\{x\in\R^n:\dist(x,A)<r\Big\}\qquad r>0
  \]
  and apply \eqref{davila 2} with $V=N_r(\Om^c)$, to find
  \[
  \k_n\,\H^{n-1}(N_r(\Om^c)\cap\pa^*E)=2\,\lim_{s\to 1^-}(1-s)\,I_s(EN_r(\Om^c),E^cN_r(\Om^c))
  \]
  and thus
  \begin{equation}
    \label{doppiocolpo}
      \k_n\,\H^{n-1}(\pa\Om\cap\pa^*E)=2\,\lim_{r\to0^+}\lim_{s\to 1^-}(1-s)\,I_s(EN_r(\Om^c),E^cN_r(\Om^c))\,.
  \end{equation}
  We have
  \begin{eqnarray}\nonumber
    I_s(EN_r(\Om^c),E^cN_r(\Om^c))&=&I_s(E,E^cN_r(\Om^c))-I_s(E\setminus N_r(\Om^c),E^cN_r(\Om^c))
    \\\nonumber
    &=&I_s(E,\Om^c)-I_s(E,E^cN_r(\Om^c)\Om)-I_s(E\setminus N_r(\Om^c),E^cN_r(\Om^c))
    \\
    &&\label{identity}
  \end{eqnarray}
  where in the last step we have use the fact that $\Om^c\subset E^c\cap N_r(\Om^c)$. We now want to estimate the two negative terms on the right-hand side of \eqref{identity}. First, since $E\subset\Om$,
  \[
  (1-s)\,I_s(E,E^cN_r(\Om^c)\Om)\le\mu^{\Om}_s(N_r(\Om^c)\cap\Om)
  \]
  and since for a.e. $r>0$ we have $\H^{n-1}(N_r(\Om^c)\cap\pa^*E)=0$ we find
  \[
  \limsup_{s\to 1^-}(1-s)\,I_s(E,E^cN_r(\Om^c)\Om)\le\k_n\,\H^{n-1}(N_r(\Om^c)\cap\Om\cap\pa^*E)\qquad\mbox{for a.e. $r>0$}\,,
  \]
  where $\H^{n-1}(N_r(\Om^c)\cap\Om\cap\pa^*E)\to 0$ as $r\to 0^+$ thanks to $\Om\cap\pa\Om=\varnothing$ and $\H^{n-1}(\pa^*E)<\infty$; summarizing,
  \begin{equation}
    \label{infsup1}
      \lim_{r\to 0^+}\lim_{s\to 1^-}(1-s)\,I_s(E,E^cN_r(\Om^c)\Om)=0\,.
  \end{equation}
  Coming now to the second term on the right-hand side of \eqref{identity}, we have
  \begin{eqnarray*}
  I_s(E\setminus N_r(\Om^c),E^cN_r(\Om^c))&=&I_s(E\setminus N_r(\Om^c),E^cN_r(\Om^c)\Om)
  +I_s(E\setminus N_r(\Om^c),E^cN_r(\Om^c)\Om^c)
  \\
  &\le&I_s(E,E^cN_r(\Om^c)\Om)
  +I_s(E\setminus N_r(\Om^c),\Om^c)
  \end{eqnarray*}
  where the first term has been addressed in \eqref{infsup1}, while the second satisfies
  \begin{eqnarray*}
  I_s(E\setminus N_r(\Om^c),\Om^c)&=&\int_{E\setminus N_r(\Om^c)}dx\int_{\Om^c}\frac{dy}{|x-y|^{n+s}}
  \le\int_{E\setminus N_r(\Om^c)}dx\int_{B_r(x)^c}\frac{dy}{|x-y|^{n+s}}
  \\
  &\le& C(n)\,|E|\,\int_r^\infty \frac{dt}{t^{1+s}}=C(n)\frac{|E|}{s\,r^s}\,,
  \end{eqnarray*}
  so that
  \begin{equation}
    \label{infsup2}
      \lim_{s\to 1^-}(1-s)\,I_s(E\setminus N_r(\Om^c),\Om^c)=0\qquad\forall r>0\,.
  \end{equation}
  By combining \eqref{doppiocolpo}, \eqref{identity}, \eqref{infsup1} and \eqref{infsup2} we deduce \eqref{davila tesi 2}.
\end{proof}

\section{The Euler-Lagrange equation}\label{section Euler lagrange}

In this section we characterize the Euler-Lagrange equation for the nonlocal capillarity energy $\E$, see Theorem \ref{thm EL equation}.

\begin{lemma}[Weak form of the Euler-Lagrange equation]\label{lemma weak EL}
  Let $K\in\K^1(n,\s,\l)$. If $\Om$ is a bounded open set with $C^1$-boundary, $g\in C^1(\R^n)$, and $E$ is a critical point of $\E+\int g$,
  then there exists a constant $c\in\R$ such that
  \begin{equation}
  \label{weak euler lagrange}
  \begin{split}
  &\iint_{E\times(E^c\cap\Om)}\Div_{(x,y)}\big(K(x-y)\,(T(x),T(y))\big)\,dxdy
  \\
  &+\s\iint_{E\times\Om^c}
  \Div_{(x,y)}\big(K(x-y)\,(T(x),T(y))\big)\,dxdy+\int_E\,\Div(g\,T)=c\,\int_E\,\Div\,T
  \end{split}
  \end{equation}
  for every $T\in C^\infty_c(\R^n;\R^n)$ with
  \[
  T\cdot\nu_\Om=0\qquad\mbox{on $\pa\Om$}\,.
  \]
\end{lemma}

\begin{proof}
  {\it Step one}: Given $T\in C^\infty_c(\R^n;\R^n)$ satisfying
  \begin{equation}
  \label{initial velocity}
  T\cdot\nu_\Om=0\quad\mbox{on $\pa\Om$}\qquad \int_E\Div\,T=0
  \end{equation}
  the flux
  $$
  \left\{\begin{split}
  &\partial_t h_t(x) = T(h_t(x))
  \\
  &h_0(x)=x
  \end{split}\right. \qquad\forall |t|<\e\,,
  $$
  generated by $T$ satisfies $h_t(\Om)=\Om$ for every $|t|<\e$ and $|h_t(E)|=|E|+{\rm O}(t^2)$. By picking any vector field
  $S\in C^\infty_c(\Om;\R^n)$ with support a positive distance from the support of $T$ and such that
  \[
  \int_E\,\Div S>0
  \]
  and by exploiting a classical argument based on the implicit function theorem (see \cite[Theorem 19.8, Step one]{maggiBOOK})
  we can find $s\in C^\infty((-\e,\e))$ with $s(0)=s'(0)=0$ such that the family of diffeomorphisms
  \begin{equation}
    \label{bolza weierstrass}
      f_t(x)=x+t\,T(x)+s(t)\,S(x)\qquad (x,t)\in\R^n\times(-\e,\e)
  \end{equation}
  satisfies $f_t(\Om)=\Om$ and $|f_t(E)|=|E|$, that is \eqref{diffeos critical point}. In particular, by assumption,
  \begin{equation}
    \label{bellostazionario}
      \frac{d}{dt}\bigg|_{t=0}\E(f_t(E))+\int_{f_t(E)}g=0\,.
  \end{equation}
  We notice that by \eqref{bolza weierstrass} (see, e.g. \cite[Lemma 17.4]{maggiBOOK})
  \begin{equation}
    \label{INI:0}
      \nabla f_t=\Id+t\,\nabla T+{\rm O}(t^2)\,,\qquad
  Jf_t=\det(\nabla f_t)=1+t\,\Div\,T+{\rm O}(t^2)\,,
  \end{equation}
  uniformly on $\R^n$ as $t\to 0$, as well as
  \begin{equation}\label{INI}
 |f_t(x)-f_t(y)|\le C|x-y|\qquad\forall x,y\in\R^n\,.
 \end{equation}
 for some~$C>0$. Moreover, if $F$ is an arbitrary Borel set and $h\in C^1(\R^n)$ then
  \begin{equation}
    \label{variazione potenziale 1}
      \frac{d}{dt}\bigg|_{t=0}\int_{f_t(F)}h=\int_F\,\Div(h\,T)
  \end{equation}
  while if $F$ is of locally finite perimeter in an open neighborhood of $\spt\,T$ and $h\in C^0(\R^n)$, then
  \begin{equation}
    \label{variazione potenziale 2}
      \frac{d}{dt}\bigg|_{t=0}\int_{f_t(F)}h=\int_{\pa^*F}\,h\,(T\cdot\nu_F)\,d\H^{n-1}\,,
  \end{equation}
  see for example \cite[Proposition 17.8]{maggiBOOK}.

  \medskip

  \noindent {\it Step two}: We assume that $K\in C^2_c(\R^n)$ and prove that
  \begin{equation}
    \label{LA:aLA1}
    \begin{split}
      \frac{d}{dt}\bigg|_{t=0} I(f_t(E),f_t(E)^c\Om)=
  \iint_{E\times(E^c\cap\Om)} \Big[
  \Div_x \big( K(x-y)T(x)\big)+\Div_y \big( K(x-y)T(y)\big)
  \Big]\,dx\,dy
  \\
  \frac{d}{dt}\bigg|_{t=0}I(f_t(E),\Om^c)=
  \iint_{E\times\Om^c} \Big[
  \Div_x \big( K(x-y)T(x)\big)+\Div_y \big( K(x-y)T(y)\big)\Big]\,dx\,dy\,.
    \end{split}
  \end{equation}
  By \eqref{bolza weierstrass} and since $s'(0)=0$ we have
  \[
  |(f_t(x)-f_t(y))-(x-y)|\le C\,t\,|x-y|\qquad\forall x,y\in\R^n\,,|t|<\e\,,
  \]
  so that if , then
  \begin{equation}\label{zeta}
  |\zeta|\ge \frac{|x-y|}{2},
  \end{equation}
  whenever $x,y\in\R^n$, $|t|<\e$, and $\zeta$ is a point lying on the segment joining~$x-y$ and~$f_t(x)-f_t(y)$.
  {F}rom \eqref{D2} and~\eqref{zeta}, $|D^2 K(\zeta)|\le C\,|x-y|^{-n-s-2}$, and thus
  \begin{equation}\label{LA:2}
  |D^2 K(\zeta)|\,|f_t(x)-f_t(y)-(x-y)|^2\le
  \frac{Ct^2 \min\{ 1,|x-y|^2\}}{|x-y|^{n+s+2}}
  \end{equation}
  for $x$, $y$, $t$ and $\zeta$ as in \eqref{zeta}. Also, since \eqref{bolza weierstrass} and $s'(0)=0$ give
  $$
  \big| f_t(x)-f_t(y)-(x-y)-t(T(x)-T(y))\big|
  \le C\,t^2 |x-y|\,,\qquad\forall x,y\in\R^n\,,|t|<\e\,,
  $$
  by using again~\eqref{D2} we find
  \begin{eqnarray*}
  && \left| \nabla K(x-y)\cdot \big(f_t(x)-f_t(y)-(x-y)\big)-
  t\nabla K(x-y)\cdot \big(T(x)-T(y)\big)
\right| \\
&\le& \frac{C}{|x-y|^{n+s+1}}\,
\big| f_t(x)-f_t(y)-(x-y)-t(T(x)-T(y))\big|\\&\le&
\frac{C t^2 \min\{ 1,|x-y|\}}{|x-y|^{n+s+1}}
\end{eqnarray*}
for every $x,y\in\R^n$, $|t|<\e$.
{F}rom this and~\eqref{LA:2},
\begin{eqnarray}\label{bellozio}
K\big( f_t(x)-f_t(y) \big)=K(x-y) +t\nabla K(x-y)\cdot \big(T(x)-T(y)\big)
+t^2 \Upsilon(x,y),
\end{eqnarray}
where here and in the rest of this proof,~$\Upsilon$ denotes a generic function (which may change
from line to line) such that
\begin{equation} \label{UPS:L}
|\Upsilon(x,y)|\le
\frac{C  \min\{ 1,|x-y|\}}{|x-y|^{n+s+1}}.
\end{equation}
By combining \eqref{INI:0} and \eqref{bellozio} we find
\begin{equation}\label{CO:0}
\begin{split}
& K\big(f_t(x)-f_t(y)\big)\,Jf_t(x) \,Jf_t(y)\\
=\;& \Big[
K(x-y) +t\nabla K(x-y)\cdot \big(T(x)-T(y)\big)
+t^2 \Upsilon(x,y)\Big]\\
&
\,\cdot\Big[1+t\Div T(x)+O(t^2)\Big]\,\Big[1+t\Div T(y)+O(t^2)\Big]
\\
=\;&  K(x-y) +t\nabla K(x-y)\cdot \big(T(x)-T(y)\big)
+tK(x,y)\big( \Div T(x)+\Div T(y)\big)+t^2 \Upsilon(x,y)\,.
\end{split}\end{equation}
Now we observe that
$$ \Div_x \big( K(x-y)T(x)\big)=\nabla K(x-y)\cdot T(x)+
K(x-y)\Div T(x).$$
Then, since~$K$ is even,
$$ \Div_y \big( K(x-y)T(y)\big)=\nabla K(x-y)\cdot T(y)+
K(x-y)\Div T(y)$$
and therefore
\begin{eqnarray*}&& \Div_x \big( K(x-y)T(x)\big)+\Div_y \big( K(x-y)T(y)\big)
\\&&\qquad= \nabla K(x-y)\cdot \big( T(x)+T(y)\big)+
K(x-y)\big(\Div T(x)+\Div T(x)\big).\end{eqnarray*}
Comparing this with~\eqref{CO:0}, we conclude that
\begin{equation*}
\begin{split}
& K\big(f_t(x)-f_t(y)\big)\,Jf_t(x) \,Jf_t(y)\\
=\;&  K(x-y) +t\Big[
\Div_x \big( K(x-y)T(x)\big)+\Div_y \big( K(x-y)T(y)\big)
\Big]
+t^2 \Upsilon(x,y).
\end{split}\end{equation*}
Consequently, by the area formula,
\begin{equation}\label{PL:AQ}
\begin{split}
&I(f_t(E),f_t(E)^c\cap\Om)=
I(E,E^c\cap\Om)\\&\qquad+
t\iint_{E\times(E^c\cap\Om)} \Big[
\Div_x \big( K(x-y)T(x)\big)+\Div_y \big( K(x-y)T(y)\big)
\Big]\,dx\,dy
\\&\qquad+t^2 \iint_{E\times(E^c\cap\Om)}\Upsilon(x,y)\,dx\,dy
\\&
I(f_t(E),\Om^c)=
I(E,\Om^c)\\&\qquad+
t\iint_{E\times\Om^c} \Big[
\Div_x \big( K(x-y)T(x)\big)+\Div_y \big( K(x-y)T(y)\big)
\Big]\,dx\,dy
\\&\qquad+t^2 \iint_{E\times\Om^c}\Upsilon(x,y)\,dx\,dy.
\end{split}\end{equation}
By~\eqref{kernelclass 0},
\eqref{UPS:L}
and~\eqref{BO:Q1} it follows that
$$ +\infty> I(E,\Om^c) \ge \iint_{{E\times \Om^c}\atop{|x-y|\le \e}}
\frac{dx\,dy}{\l\,|x-y|^{n+s}}
\ge \frac{1}{C\,\l}
\iint_{{E\times \Om^c}\atop{|x-y|\le \e}}\Upsilon(x,y)\,dx\,dy$$
and thus
$$ \iint_{E\times \Om^c} \Upsilon(x,y)\,dx\,dy<+\infty.$$
Similarly (using~\eqref{BO:Q2}
in lieu of~\eqref{BO:Q1}), we obtain that
$$ \iint_{E\times(E^c\cap \Om)} \Upsilon(x,y)\,dx\,dy<+\infty.$$
Accordingly, we find from~\eqref{PL:AQ} that
\begin{equation*}
\begin{split}
&I(f_t(E),f_t(E)^c\cap\Om)=
I(E,E^c\cap\Om)\\&\qquad+
t\iint_{E\times(E^c\cap\Om)} \Big[
\Div_x \big( K(x-y)T(x)\big)+\Div_y \big( K(x-y)T(y)\big)
\Big]\,dx\,dy+O(t^2)
\\\;&
I(f_t(E),\Om^c)=
I(E,\Om^c)\\&\qquad+
t\iint_{E\times\Om^c} \Big[
\Div_x \big( K(x-y)T(x)\big)+\Div_y \big( K(x-y)T(y)\big)
\Big]\,dx\,dy+O(t^2)\,.
\end{split}\end{equation*}
This completes the proof of \eqref{LA:aLA1}, thus of step two.

\medskip

\noindent {\it Step three}: We now claim that \eqref{LA:aLA1} holds with $K\in\K_*(n,\s,\l,\e)$
in place of a generic $K\in C^2_c(\R^n)$. For each~$\delta\in(0,1/2)$, let~$\eta_\delta \in C^\infty([0,+\infty))$ be
 such that~$\eta_\delta=1$
  in~$[0,\delta]\cup [1/\delta,+\infty)$, $\eta_\delta=0$
  in~$[2\delta, 1/(2\delta)]$, $|\eta'_\delta|\le 4/\delta$, and~$\eta_\delta\to0$ monotonically as~$\delta\to0$, and set
  \begin{equation}
    \label{Kdelta def}
      K_\de=(1-\eta_\de)\,K\,.
  \end{equation}
  If we let
  \[
  \phi_\delta(t):=\E_\de(f_t(E))\qquad\phi(t):=\E(f_t(E))
  \]
  then by monotone convergence, $\phi_\de(t)\to\phi(t)$ as $\de\to 0^+$ for every $|t|<\e$, where $\phi_\de$ and $\phi$ are smooth
  functions by the area formula (and since $I(E,E^c\Om),I(E,\Om^c)<\infty$).
  On noticing that
  \[
  \frac{\pa f_t}{\pa t}(x)=T(x)+s'(t)\,S(x)=:T_t(x)
  \]
  by \eqref{LA:aLA1} we have
  \begin{equation}
    \label{thanksto}
      \phi_\de'(t)=\bigg(\iint_{E\times(E^c\cap\Om)}+\s\,\iint_{E\times\Om^c}\bigg)
  \Big[
  \Div_x \big( K_\de(x-y)T_t(x)\big)+\Div_y \big( K_\de(x-y)T_t(y)\big)
  \Big]\,dxdy\,.
  \end{equation}
  We now claim that
  \begin{equation}
    \label{belclaim}
      \phi_\de'(t)\to
  \bigg(\iint_{E\times(E^c\cap\Om)}+\s\,\iint_{E\times\Om^c}\bigg)
  \Big[
  \Div_x \big( K(x-y)T_t(x)\big)+\Div_y \big( K(x-y)T_t(y)\big)
  \Big]\,dxdy
  \end{equation}
  uniformly on $|t|<\e$ as $\de\to 0^+$. By applying the mean value theorem to $\phi_\de$ and since
  $\phi_\de\to\phi$ as $\de\to 0^+$ pointwise, this will imply that
  \[
  \phi'(0)=\bigg(\iint_{E\times(E^c\cap\Om)}+\s\,\iint_{E\times\Om^c}\bigg)
  \Big[
  \Div_x \big( K(x-y)T(x)\big)+\Div_y \big( K(x-y)T(y)\big)
  \Big]\,dxdy
  \]
  as required. To prove \eqref{belclaim} we just notice that
  \begin{eqnarray*}
  &&\Div_x \big( K_\de(x-y)T_t(x)\big)+\Div_y \big( K_\de(x-y)T_t(y)\big)
  \\
  &&=
  K_\de(x-y)\,\big(\Div T_t(x)+\Div\,T_t(y)\big)+\nabla K_\de(x-y)\cdot (T_t(x)-T_t(y))
  \end{eqnarray*}
  where $|T_t(x)-T_t(y)|\le C\,|x-y|$ for every $x,y\in\R^n$ and $|t|<\e$, so that \eqref{D2} gives
  \[
  \Big|\Div_x \big( K_\de(x-y)T_t(x)\big)+\Div_y \big( K_\de(x-y)T_t(y)\big)\Big|\le \frac{C}{|x-y|^{n+s}}\le C\,K(x-y)\,,
  \]
  and, in conclusion, \eqref{belclaim} holds by dominated convergence and thanks to
  $I(E,E^c\Om),I(E,\Om^c)<\infty$ (recall \eqref{BO:Q1} and \eqref{BO:Q2}).

  \medskip

  \noindent {\it Step four}: Let us consider the linear functional on $T\in C^\infty_c(\R^n;\R^n)$
  defined by
  \[
  \L(T)=\bigg(\iint_{E\times(E^c\cap\Om)}+\s\,\iint_{E\times\Om^c}\bigg)
  \Div_{(x,y)}\Big( K(x-y)\,\big(T(x),T(y)\big)\Big)\,dxdy+\int_E\,\Div(g\,T)\,.
  \]
  By combining \eqref{bellostazionario}, \eqref{variazione potenziale 1} and step three
  we find that $\L(T)=0$ whenever $T$ satisfies \eqref{initial velocity}.
  If $T_1, T_2\in C^\infty(\R^n;\R^n)$ have disjoint supports and are such that
  \[
  T_1\cdot\nu_\Om=T_2\cdot\nu_\Om=0\quad\mbox{on $\pa\Om$}\qquad
  \int_E\,\Div T_2\ne 0\,,
  \]
  then
  \[
  T=T_1-\frac{\int_E\Div\,T_1}{\int_E\Div\,T_2}\,T_2
  \]
  is admissible in \eqref{initial velocity}, and thus satisfy $\L(T)=0$. Thus $\L(T_1)/\int_E\Div T_1=\L(T_2)/\int_E\Div T_2$, and
  the proof is completed by the arbitrariness of $T_1$ and $T_2$.
\end{proof}

In passing from Lemma \ref{lemma weak EL} to Theorem \ref{thm EL equation} we shall need the following proposition.

\begin{proposition}\label{proposition curvatura continua}
  If $\s\in(-1,1)$ and $K\in \K^1(n,s,\l)$, then for every $E\subset\Om$ the function
  \[
  \HH^{K,\s,\Om}_{\pa E}(x):={\rm p.v.}\,\int_{\R^n}K(x-y)\,
  \big( 1_{E^c\cap\Om}(y)+\s\,1_{\Om^c}(y)-1_E(y)\big)\,dy\qquad x\in\pa E
  \]
  is continuous on $\Om\cap{\rm Reg}_E$ with
  \begin{equation}
    \label{Hdelta uniforme a H}
      \HH^{K_\de,\s,\Om}_{\pa E}\to  \HH^{K,\s,\Om}_{\pa E}\qquad\mbox{as $\de\to 0^+$}
  \end{equation}
  uniformly on compact subsets of $\Om\cap{\rm Reg}_E$. Here, $K_\de$ is defined as in \eqref{Kdelta def}.
\end{proposition}

\begin{proof}
  Since $K_\de\in C^1_c(\R^n)$ we definitely have
  \begin{equation}
    \label{mean curvature identity delta}
      \HH^{K_\de,\s,\Om}_{\pa E}(x)=\HH^{K_\de}_{\pa E}(x)-(1-\s)\,\int_{\Om^c}K_\de(x-y)\,dy\qquad\forall x\in\pa E\,,
  \end{equation}
  see \eqref{nonlocal mean curvature} for the definition of $\HH^{K_\de}_{\pa E}$. It is shown in \cite[Proposition 6.3]{F2M3} that the continuous functions $\{\HH^{K_\de}_{\pa E}\}_{\de}$ converge uniformly on compact subsets of $\Om\cap{\rm Reg}_E$ to $\HH^K_{\pa E}$. An identical argument leads to obtain \eqref{Hdelta uniforme a H}, proves the continuity of $\HH^{K,\s,\Om}_{\pa E}$ on ${\rm Reg}_E\cap\Om$.
\end{proof}

\begin{proof}[Proof of Theorem~\ref{thm EL equation}] Let $K_\de\in C^2_c(\R^n)$ be defined as in \eqref{Kdelta def}. As soon as $E$ has finite perimeter, one has (by~\cite[Formula~(15.11)]{maggiBOOK})
$$
\int_{E} \Div_x \big( K_\de(x-y)T(x)\big)\,dx
= \int_{\pa^* E} K_\de(x-y)T(x)\cdot\nu_{E}\,
d\H^{n-1}_x$$
where $\pa^*E$ denotes the reduced boundary of $\pa E$ and $\nu_E$ its measure-theoretic outer unit normal. In particular,
for any set~$F$ that does not intersect~$E$ we find
\begin{equation}\label{09:AS}
\begin{split}
\int_{E\times F} \Div_x \big( K_\de(x-y)T(x)\big)\,dx\,dy
=\;\int_F
\left( \int_{\pa^* E}
K_\de(x-y)T(x)\cdot\nu_{E}\,
d\H^{n-1}_x\right)\,
dy.\end{split}\end{equation}
Similarly, for any set~$F$ that does not intersect~$E$,
$$ \int_{F} \Div_y \big( K_\de(x-y)T(y)\big)\,dy
= \int_{\pa^* F} K_\de(x-y)T(y)\cdot\nu_{F}\,
d\H^{n-1}_y$$
and therefore, integrating in~$E$ and changing the names of the variables,
\begin{eqnarray*}
\iint_{E\times F} \Div_y \big( K_\de(x-y)T(y)\big)\,dx\,dy
&=& \int_E\left(
\int_{\pa^* F} K_\de(x-y)T(y)\cdot\nu_{F}
\,d\H^{n-1}_y\right)\,dx
\\&=& \int_E
\left(\int_{\pa^* F}
 K_\de(x-y)T(x)\cdot\nu_{F}\,d\H^{n-1}_x
\right)\,dy.
\end{eqnarray*}
Using this formula and~\eqref{09:AS}
with~$F=E^c\cap\Om$, we obtain that
\begin{equation}\label{09:AB1}\begin{split}
&\iint_{E\times (E^c\cap\Om)}
\Div_{(x,y)} \big( K_\de(x-y)(T(x),T(y))\big)\,dx\,dy\\
=\;& \int_{E^c\cap\Om}
\left(
\int_{\pa^* E}
K_\de(x-y)T(x)\cdot\nu_{E}\,d\H^{n-1}_x
\right)\,
dy\\&\qquad+\int_E\left(
\int_{\pa^* (E^c\cap\Om)}
K_\de(x-y)T(x)\cdot\nu_{E^c\cap\Om}\,d\H^{n-1}_x
\right)\,dy\\
=\;& \int_{\Omega\cap \pa^* E}
T(x)\cdot\nu_{E} \left(\int_{\R^n} K_\de(x-y)\,
\big( 1_{E^c\cap\Om}(y)-1_E(y)\big)\,dy\right)\,d\H^{n-1}_x\,,
\end{split}
\end{equation}
and analogously
\begin{equation}\label{09:AB2}\begin{split}
&\iint_{E\times \Om^c}\Div_{(x,y)} \big( K_\de(x-y)(T(x),T(y))\big)\,dx\,dy
\\
=\;& \int_{\Omega\cap \pa^* E}
T(x)\cdot\nu_{E} \left(\int_{\R^n}K_\de(x-y)\,
1_{\Om^c}(y)\,dy\right)\,d\H^{n-1}_x.\end{split}
\end{equation}
In particular,
\begin{equation}\label{to the limit}
  \begin{split}
    &\bigg(\iint_{E\times(E^c\cap\Om)}+\s\iint_{E\times \Om^c}\bigg)\Div_{(x,y)} \big( K_\de(x-y)(T(x),T(y))\big)\,dx\,dy
\\&=\int_{\Om\cap\pa^*E}(T\cdot\nu_E)\,\HH^{K_\de,\s,\Om}_{\pa E}\,d\H^{n-1}\,.
  \end{split}
\end{equation}
Let us now fix $x\in\Om\cap{\rm Reg}_E$ and $T\in C^1_c(B_{{{\varrho}}}(x)\cap\Om)$ with ${{\varrho}}>0$ such that $B_{2\,{{\varrho}}}(x)\cap\pa E\subset\Om\cap{\rm Reg}_E$. In this way, $\HH^{K_\de,\s,\Om}_{\pa E}$ converges uniformly to $\HH^{K,\s,\Om}_{\pa E}$ on $\spt\,T$ and thus the right-hand side of \eqref{to the limit} converges to $\int_{B_{{{\varrho}}}(x)\cap\pa E}(T\cdot\nu_E)\,\HH^{K,\s,\Om}_{\pa E}$. Since we have already shown in the proof of Lemma \ref{lemma weak EL} that in the limit $\de\to 0^+$ we can take replace $K_\de$ by $K$  on the left-hand side of \eqref{to the limit}, we conclude that
\begin{equation}\nonumber
  \begin{split}
&    \bigg(\iint_{E\times(E^c\cap\Om)}+\s\iint_{E\times \Om^c}\bigg)\Div_{(x,y)} \big( K(x-y)(T(x),T(y))\big)\,dx\,dy
\\&=\int_{B_{{{\varrho}}}(x)\cap\pa E}(T\cdot\nu_E)\,\HH^{K,\s,\Om}_{\pa E}\,d\H^{n-1}\,.
  \end{split}
\end{equation}
for every $x\in\Om\cap{\rm Reg}_E$ and $T\in C^1_c(B_{{{\varrho}}}(x)\cap\Om)$, for ${{\varrho}}>0$ depending on $x$. By combining this identity with \eqref{weak euler lagrange}, $\int_{E}\Div(T\,g)=\int_{B_{{{\varrho}}}(x)\cap\pa E}g\,(T\cdot\nu_E)$, and the arbitrariness of $T$, we finally deduce \eqref{stationary set}.
\end{proof}

\section{Nonlocal Young's law}\label{section nonlocal young}
This section addresses the proof of Theorem \ref{thm nonlocal young}. We premise a simple technical lemma. Here, we decompose $x\in\R^n$ as $x=(x',x_n)\in\R^{n-1}\times\R$ and set
\[
\C=\{x\in\R^n:|x'|<1\,,|x_n|<1\}\qquad{\mbox{and}}\qquad \D=\{z\in\R^{n-1}:|z|<1\}\,.
\]

\begin{lemma}\label{lemma carino}
  Let $\l\ge 1$, $s\in(0,1)$ and $\a\in(s,1)$. If $\{F_k\}_{k\in\N}$ is a sequence of Borel sets in $\R^n$ with $0\in\pa F_k$,
  \[
  \mbox{$F_k\to F$ in $L^1_{{\rm loc}}(\R^n)$ for some $F\subset\R^n$}\,,
  \]
  and, for some functions $u_k, u\in C^{1,\a}(\R^{n-1})$,
  \[
  \C\cap F_k=\Big\{x\in\C:x_n\le u_k(x')\Big\}\qquad{\mbox{and}}\qquad   \lim_{k\to\infty}\|u_k-u\|_{C^{1,\a}(\D)}=0\,,
  \]
  then
  \[
  \lim_{k\to\infty}\HH^{K_k}_{\pa F_k}(0)=\HH^K_{\pa F}(0)
  \]
  whenever $\{K_k\}_{k\in\N}$ and $K$ are kernels in $\K(n,s,\l,0)$ with $K_k\to K$ pointwise in $\R^n\setminus\{0\}$.
\end{lemma}

\begin{proof}
  Up to rigid motions we may assume without loss of generality that $0\in\pa F$ (so that $u(0)=u_k(0)=0$) and that $\nabla u_k(0)=\nabla u(0)=0$. Since $u\in C^{1,\a}(\D)$ and $u_k\to u$ in $C^{1,\a}(\D)$ we can find $\g>0$ such that
  \[
  \max\{|u_k(z)|,|u(z)|\}\le\g|z|^{1+\a}\qquad\forall z\in\D\,,k\in\N\,.
  \]
  If we let
  \[
  P_{\e,\g}=\big\{x\in B_\e:|x_n|<\g\,|x'|^{1+\a}\big\}\qquad\e\in(0,1)\,,
  \]
  then $|z|^{-n-s}\in L^1(P_{\e,\g}\cup (B_\e)^c)$ and thus
  \begin{eqnarray*}
  \HH^{K_k}_{\pa F_k}(0)=\int_{(B_\e)^c\cup P_{\e,\g}}(1_{F_k}^c-1_{F_k})\,K_k
  \qquad
  \HH^{K}_{\pa F}(0)=\int_{(B_\e)^c\cup P_{\e,\g}}(1_{F}^c-1_{F})\,K\,.
  \end{eqnarray*}
Since $(1_{F_k}^c-1_{F_k})\to (1_{F}^c-1_{F})$ a.e. on $\R^n$ we conclude by dominated convergence that $\HH^{K_k}_{\pa F_k}(0)\to \HH^K_{\pa F}(0)$.
\end{proof}

\begin{proof}[Proof of Theorem \ref{thm nonlocal young}] {\it Step one}: We start proving the validity of \eqref{young law nonlocal general}. Let us fix $x_0\in\pa\Om\cap{\rm Reg}_E$ so that $x_0$ is a boundary point of the manifold with boundary $B_{{{\varrho}}}(x_0)\cap\pa E$. Consider a sequence $\{x_k\}_{k\in\N}\subset\Om\cap{\rm Reg}_E$ such that $x_k\to x_0$, and set
\[
r_k=|x_k-x_0|\qquad v_k=\frac{x_k-x_0}{r_k}\qquad E^{x_0,r_k}=\frac{E-x_0}{r_k}\qquad \Om^{x_0,r_k}=\frac{\Om-x_0}{r_k}\,.
\]
We recall that, by \eqref{stationary set},
\begin{equation}
  \label{recall that}
  \HH^K_{\pa E}(x_k)-(1-\s)\,\int_{\Om^c}K(x_k-y)\,dy+g(x_k)=c
\end{equation}
for a constant $c$ independent of $k$. We have that
\[
\mbox{$\Om^{x_0,r_k}\to H$ and $E^{x_0,r_k}\to V\cap H$ in $L^1_{{\rm loc}}(\R^n)$}
\]
where $H$ and $V$ are suitable half-spaces in $\R^n$ so that
\[
\nu_\Om(x_0)=\nu_H(0)\qquad \nu_V(0)=\lim_{k\to\infty}\nu_E(x_k)=:\nu_V(0)\,.
\]
Up to extracting subsequences, we have that $v_k\to v$ for some $v\in S^{n-1}$. We can use the change of variables $y=x_0+r_k\,z$ to find
\begin{eqnarray}\label{curvature blow ups}
\HH^{K}_{\pa E}(x_k)&=&\int_{\R^n}K(x_k-y)\,
\big( 1_{E^c}(y)-1_E(y)\big)\,dy
\\
&=&r_k^{-s}\int_{\R^n}r_k^{n+s}\,K(x_k-x_0-r_k\,z)\,
\big( 1_{(E^{x_0,r_k})^c}(z)-1_{E^{x_0,r_k}}(z)\big)\,dz
\end{eqnarray}
Now, since $\{x_k\}_{k\in\N}\subset\Om\cap{\rm Reg}_E$, we can find rigid motions $Q_k:\R^n\to\R^n$ and functions $u_k\in C^{1,\a}(\R^{n-1})$ such that if we set
\[
F_k=Q_k(E^{x_0,r_k}-v_k)
\]
then $0\in\pa F_k$ and
\[
\C\cap F_k=\Big\{x\in\C:x_n\le u_k(x')\Big\}\,.
\]
Notice that
\[
\mbox{$F_k\to F=H\cap V$ in $L^1_{{\rm loc}}(\R^n)$}
\]
with $u_k\to u$ in $C^{1,\a}(\D)$ for a linear function $u:\R^{n-1}\to\R$. If we set
\[
K_k(\zeta)=r_k^{n-s}\,K(r_k\,\zeta)\qquad\zeta\in\R^n\setminus\{0\}\,,
\]
then by \eqref{curvature blow ups} we get
\[
\HH^{K}_{\pa E}(x_k)=r_k^{-s}\,\HH^{K_k}_{\pa F_k}(0)\,.
\]
Since $K_k\to K^*$ pointwise in $\R^n\setminus\{0\}$, by Lemma \ref{lemma carino} we find
\[
\lim_{k\to\infty}r_k^s\,\HH^{K}_{\pa E}(x_k)=\HH^{K^*}_{\pa (H\cap V)}(v)\,,
\]
and since $r_k^s\,g(x_k)\to 0$ (indeed $x_k\to x_0$ and $g$ is locally bounded), \eqref{recall that} implies
\[
\HH^{K^*}_{\pa (H\cap V)}(v)-(1-\s)\,\lim_{k\to\infty}r_k^s\,\int_{\Om^c}K(x_k-y)\,dy=0\,.
\]
By the change of variable $y=x_0+r_k\,z$,
\[
\int_{\Om^c}K(x_k-y)\,dy=r_k^{-s}\,\int_{(\Om^{x_0,r_k})^c}\,r_k^{n+s}\,K\big(r_k(v_k-z)\big)\,dz
\]
where
\[
\lim_{k\to\infty}\int_{(\Om^{x_0,r_k})^c}\,r_k^{n+s}\,K\big(r_k(v_k-z)\big)\,dz=\int_{H^c}\,K^*(v-z)\,dz\,.
\]
We have thus proved that
\begin{equation}\label{trieste}
\HH^{K^*}_{\pa (H\cap V)}(v)-(1-\s)\,\int_{H^c}\,K^*(v-z)\,dz=0\,,\qquad\forall v\in H\cap\pa V\,,
\end{equation}
that is \eqref{young law nonlocal general}.

\medskip

\noindent {\it Step two}: We now assume that $K=K^\e_s$ for some $\e>0$, so that $K^*=K_s$. Up to a rigid motion we can assume that $H$ and $V$ satisfy
\begin{eqnarray*}
H&=&\{x\in\R^n:x_n>0\}
\\
H\cap V&=&\Big\{x\in \R^n:\mbox{$x_n>0$ and $\cos\a\,x_n=\sin\a\,x_1$ for some $\a\in(0,\theta)$}\Big\}=:J_\theta\,,
\end{eqnarray*}
for some $\theta\in(0,\pi)$. Since \eqref{trieste} is $-s$ homogeneous in $|v|$, we find that \eqref{trieste} is equivalent to
\begin{equation}\label{aJAH:8A}
\int_{\R^n}\frac{(1_{J_\theta^c\cap H}+
\s\,1_{H^c}-1_{J_\theta})(z)}{|e(\theta)-z|^{n+s}}\,dz=0
\end{equation}
where
\[
e(\theta)=\cos\,\theta\,e_1+\sin\theta\,e_n\,.
\]
In this step we show that there exists a unique $\theta=\theta(n,s,\s)\in(0,\pi)$ such that \eqref{aJAH:8A} holds -- so that, correspondingly,
\[
\nu_E(x_0)\cdot\nu_{\Om}(x_0)=\nu_V(0)\cdot\nu_H(0)=\cos\big(\pi-\theta(n,s,\s)\big)
\]
and \eqref{young law fractional} holds -- and that the function $\s\in(-1,1)\mapsto\theta(n,s,\s)$ is strictly increasing with
\begin{equation}\label{theta limits}
\theta(n,s,0)=\frac\pi2\,,\qquad\lim_{\s\to (-1)^+}\theta(n,s,\s)=0\,,\qquad\lim_{\s\to 1^-}\theta(n,s,\s)=\pi\,.
\end{equation}

We first notice that we do not need to specify the integral in~\eqref{aJAH:8A} in the principal value sense as there always is a ball centered at $e(\theta)$ with one half of it contained in $J_\theta$, the other half contained in $J_\theta^c\cap H$. It is also geometrically evident (see Figure \ref{fig piani}) that the choice $\s=0$, $\theta=\pi/2$ solves \eqref{aJAH:8A} and that if a pair $(\s,\theta)$ satisfies \eqref{aJAH:8A} then (i) $\theta\in(0,\pi/2)$ if and only if $\s\in(-1,0)$; (ii) $\theta\in(\pi/2,\pi)$ if and only if $\s\in(0,1)$; (iii) if $\theta\in[\pi/2,\,\pi)$, then $(-\s,\pi-\theta)$ also solves \eqref{aJAH:8A}.

We are thus left to show that $\s\in(-1,0)$ there exists a unique $\theta\in(0,\pi/2)$ (also depending on $n$ and $s$) such that  \eqref{aJAH:8A} holds, and that the correspondence $\s\in(-1,0)\mapsto\theta(n,s,\s)$ is strictly increasing and satisfies $\theta(n,s,(-1)^+)=0$. To prove this, let us notice that having restricted $\s\in(-1,0)$, we can directly consider \eqref{aJAH:8A} with $\theta\in(0,\pi/2)$. Since in this case the reflection of $J_\theta$ with respect to the hyperplane containing $H\cap\pa J_\theta$ is entirely contained in $J_\theta^c\cap H$, \eqref{aJAH:8A} turns out to the be equivalent to
\begin{equation}\label{def of theta}
\int_{\R^n}\frac{(1_{L_\theta}+\s\,1_{H^c})(z)}{|e(\theta)-z|^{n+s}}\,dz=0
\end{equation}
where $L_\theta$ is equal to $H$ minus the union of $J_\theta$ with its reflection with respect to  the hyperplane containing $H\cap\pa J_\theta$. With
\begin{figure}
  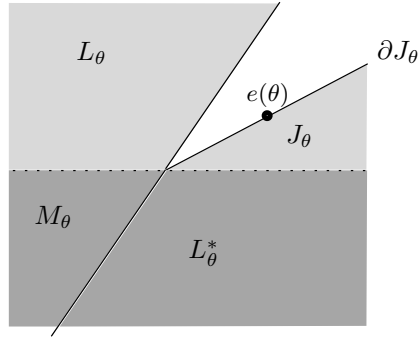\caption{{\small The cones $J_\theta$, $L_\theta$ and $M_\theta$ when $\theta\in(0,\pi/2)$.}}\label{fig piani}
\end{figure}
Figure \ref{fig piani} in mind, now let $L^*_\theta$ be the
reflection of $L_\theta$ with respect to the hyperplane containing $H\cap\pa J_\theta$,
so that $L^*_\theta$ is contained in $H^c$, and let $M_\theta=H^c\cap (L^*_\theta)^c$. As $H^c=L^*_\theta\cup M_\theta$ with $L^*_\theta\cap M_\theta=\varnothing$ and since $L^*_\theta$ is mapped into $L_\theta$ by an isometry keeping the distance from $e(\theta)$ invariant, we get
\begin{equation}
  \label{star}
  \int_{\R^n}\frac{(1_{L_\theta}+\s\,1_{H^c})(z)}{|e(\theta)-z|^{n+s}}\,dz=(1+\s)\,
\int_{L_\theta}\frac{dz}{|e(\theta)-z|^{n+s}}+
\s\,\int_{M_\theta}\frac{dz}{|e(\theta)-z|^{n+s}}
\end{equation}
and thus, by \eqref{def of theta}, we conclude that \eqref{aJAH:8A} holds for some $\theta\in(0,\pi/2)$ if and only if
\begin{equation}\label{equation:LIM}
\int_{M_\theta}\frac{dz}{|e(\theta)-z|^{n+s}}=-\Big(1+\frac1{\s}\Big)\,\int_{L_\theta}\frac{dz}{|e(\theta)-z|^{n+s}}\,.
\end{equation}
Let us set
\[
a(\theta)=\int_{M_\theta}\frac{dz}{|e(\theta)-z|^{n+s}}\qquad b(\theta)=\int_{L_\theta}\frac{dz}{|e(\theta)-z|^{n+s}}\,.
\]
Clearly $a(\theta)$ is strictly increasing on $(0,\pi/2)$, with $a(0)=0$ and $a(\pi/2)<\infty$: indeed
\[
a(\theta)=\int_{U_\theta}\frac{dz}{|z-e_n|^{n+s}}\qquad U_\theta=\Big\{x\in\R^n:x_n<0\,,|x_1|<|x_n|\,\tan\theta\Big\}\,,
\]
where the latter function is trivially increasing as $|U_{\theta_2}\setminus U_{\theta_1}|>0$ whenever $0<\theta_1<\theta_2<\pi/2$. At the same time
$b(\theta)$ is strictly decreasing with $b(0^+)=+\infty$ and $b((\pi/2)^-)=0^+$. This is seen as while $\theta$ increases from $0$ to $\pi/2$, the region $L_\theta$ is strictly decreasing from $H$ to the empty set, while the distance between the singularity $e(\theta)$ and $L_\theta$ is strictly increasing. In conclusion
\[
\theta\in\Big(0,\frac{\pi}2\Big)\mapsto \frac{a(\theta)}{b(\theta)}
\]
is a strictly increasing function on $(0,\pi/2)$ with limit $0$ as $\theta\to 0^+$ and limit $+\infty$ as $\theta\to(\pi/2)^-$. Moreover,
\[
\s\in(-1,0)\mapsto-\Big(1+\frac1{\s}\Big)
\]
is a strictly increasing function on $(-1,0)$ with limit $0$ as $\s\to(-1)^+$ and limit $+\infty$ as $\s\to 0^-$. In conclusion, for every $\s\in(-1,0)$ there exists a unique $\theta=\theta(n,s,\s)\in(0,\pi/2)$ such that \eqref{def of theta} holds. The resulting map $\s\in(-1,0)\mapsto\theta(n,s,\s)$ is strictly increasing and satisfies the first two properties in \eqref{theta limits}. This completes the proof of step two.

\medskip

\noindent {\it Step three}: We conclude the proof of the theorem by showing that $\theta(n,s,\s)=\theta(s,\s)$ with
\[
\lim_{s\to 1^-}\cos(\pi-\theta(s,\s))=\s\,,\qquad\forall\s\in(-1,1)\,.
\]
To this end, let us first go back to \eqref{def of theta}, and notice that
\[
(1_{L_\theta}+\s\,1_{H^c})(z)=f(z_1,z_n)
\]
so that if $n\ge 3$, then \eqref{def of theta} takes the form
\begin{equation}
  \label{def of theta 2}
  \int_{\R}dz_1\int_\R f(z_1,z_n)\,dz_n\int_{\R^{n-2}}\frac{dw}{(\ell^2+|w|^2)^{(n+s)/2}}=0\qquad
\end{equation}
where we have set
\[
\ell(z_1,z_n)=\sqrt{(z_1-\cos\theta)^2+(z_n-\sin\theta)^2}\,.
\]
Now, in polar coordinates,
\[
\int_{\R^{n-2}}\frac{dw}{(\ell^2+|w|^2)^{(n+s)/2}}=(n-2)\,\om_{n-2}\,\int_0^\infty\frac{r^{n-3}\,dr}{(\ell^2+r^2)^{(n+s)/2}}
\]
where, by scaling,
\[
\int_0^\infty\frac{r^{n-3}\,dr}{(\ell^2+r^2)^{(n+s)/2}}=\frac{C(n,s)}{\ell^{2+s}}\,.
\]
%
%
By taking \eqref{def of theta 2} into account, the definition \eqref{def of theta} of $\theta$ boils down to
\[
\int_{\R}dz_1\int_\R \frac{f(z_1,z_n)}{\ell^{2+s}}\,dz_n=0\,,
\]
which is actually equivalent to \eqref{def of theta} in the case $n=2$. This proves that $\theta(n,s,\s)=\theta(2,s,\s)$ for every $n\ge 3$. We thus plainly set $\theta=\theta(s,\s)$ and then turn to the proof of $\cos(\pi-\theta(s,\s))\to \s$ as $s\to 1^-$.

By exploiting the symmetries of $\theta(s,\s)$ in $\s$, it suffices to consider the case when $\s\in(-1,0)$ (and thus $\theta\in(0,\pi/2)$). It is then convenient to rewrite \eqref{star} by using $\int_{L_\theta}=\int_{L_\theta\cup M_\theta}-\int_{M_\theta}$, to find that
\[
1+\s=\frac{\int_{M_\theta}|z-e(\theta)|^{-(2+s)}dz}{\int_{L_\theta\cup M_\theta}|z-e(\theta)|^{-(2+s)}dz}\,.
\]
Notice that $L_\theta\cup M_\theta$ is an half-plane lying at distance $\sin\theta$ from $e(\theta)$. Hence,
\begin{eqnarray*}
\int_{L_\theta\cup M_\theta}\frac{dz}{|z-e(\theta)|^{2+s}}=\int_{\{y_2<0\}}\frac{dy}{|y-\sin\theta e_2|^{2+s}}
=\frac1{(\sin\theta)^s}\,\int_{\{x_2<0\}}\frac{dx}{|x-e_2|^{2+s}}\,.
\end{eqnarray*}
At the same time, by a counter-clockwise rotation around the origin of angle $(\pi/2)-\theta$, which thus maps $e(\theta)=\cos\theta e_1+\sin\theta\,e_2$ into $e_2$, we find
\[
\int_{M_\theta}\frac{dz}{|z-e(\theta)|^{2+s}}=\int_{\Gamma_\theta}\frac{dx}{|x-e_2|^{2+s}}
\]
where we have set
\[
\Gamma_\theta=\Big\{w\in\R^2:x_2<0\,,-\theta<\arctan\Big(-\frac{x_1}{x_2}\Big)<\theta\Big\}\,,
\]
see
\begin{figure}
  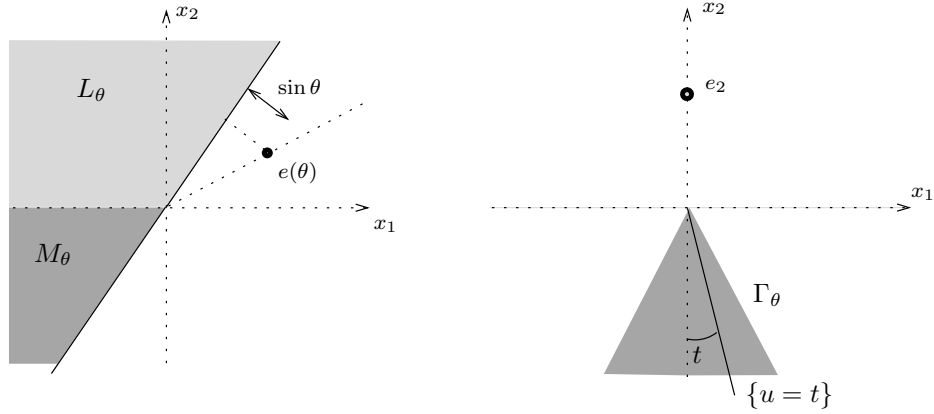\caption{{\small Notation used in computing the limit of $\theta(s,\s)$ as $s\to 1^-$.}}\label{fig piani2}
\end{figure}
Figure \ref{fig piani2}. Putting everything together we find that $\theta=\theta(s,\s)$ satisfies
\begin{eqnarray}
  \label{ictp}
  \frac{(1+\s)}{(\sin\theta)^s}\int_{\Gamma_{\pi/2}}\frac{dx}{|x-e_2|^{2+s}}=\int_{\Gamma_\theta}\frac{dx}{|x-e_2|^{2+s}}\,,
\end{eqnarray}
(indeed $\Gamma_{\pi/2}=\{x_2<0\}$). We now consider the function $u:\{x_2<0\}\to(-\pi/2,\pi/2)$ defined by
\[
u(x)=\arctan\Big(-\frac{x_1}{x_2}\Big)
\]
and notice that $u(x)$ is a locally Lipschitz on $\{x_2<0\}$ with $\Gamma_\theta=\{-\theta<u<\theta\}$ and
\[
|\nabla u|=\frac1{|x|}\,.
\]
By the Coarea formula for every Borel function $g:\{x_2<0\}\to[0,\infty]$ we have
\[
\int_{\{x_2<0\}}g(x)\,|\nabla u(x)|\,dx=\int_{-\pi/2}^{\pi/2}dt\int_{\{u=t\}}g(x)\,d\H^{n-1}(x)
\]
so that, by choosing
\[
g(x)=\frac{1_{\Gamma_\theta}(x)}{|\nabla u(x)|\,|x-e_2|^{2+s}}
\]
we get
\[
\int_{\Gamma_\theta}\frac{dx}{|x-e_2|^{2+s}}=\int_{-\theta}^\theta\,dt\int_{\{u=t\}}\frac{|x|}{|x-e_2|^{2+s}}\,d\H^1_x
=2\int_0^\theta\,dt\int_{\{u=t\}}\frac{|x|}{|x-e_2|^{2+s}}\,d\H^1_x\,.
\]
Now, if $t\in(0,\pi/2)$, then $\{u=t\}$ is the half-line $\{x\in\R^2:x_1>0\,, x_2=-(\tan t)\,x_1\}$ so that
\[
x_1=|x|\,\sin t\qquad x_2=-|x|\cos t\qquad\forall x\in\{u=t\}\,.
\]
Hence, setting $|x|=r$ we find
\[
\int_{\{u=t\}}\frac{|x|}{|x-e_2|^{2+s}}\,d\H^1_x=
\int_0^\infty\,\frac{r\,dr}{(r^2+2r\cos t+1)^{(2+s)/2}}\,.
\]
By dominated convergence
\begin{eqnarray*}
\lim_{s\to 1^-}\int_0^\theta\,dt\int_0^\infty\,\frac{r\,dr}{(r^2+2r\cos(t)+1)^{(2+s)/2}}&=&
\int_0^\theta\,dt\int_0^\infty\,\frac{r\,dr}{(r^2+2r\cos t+1)^{3/2}}
\\
&=&\int_0^\theta\frac{dt}{1+\cos t}=\frac{\sin\theta}{1+\cos\theta}\,,
\end{eqnarray*}
where we have used
\[
\int\,\frac{r\,dr}{(r^2+2r\cos t+1)^{3/2}}=-\frac1{\sin^2t}\frac{1+r\,\cos t}{\sqrt{1+2 r\cos t+r^2}}+{\rm const.}
\]
In summary, by taking the limit as $s\to 1^-$ in \eqref{ictp} we find
\[
\frac{1+\s}{\sin(\theta(1,\s))}=\frac{\sin(\theta(1,\s))}{1+\cos(\theta(1,\s))}\,,
\]
which gives $\s=-\cos(\theta(1,\s))=\cos(\pi-\theta(1,\s))$. This completes the proof of Theorem \ref{thm nonlocal young}.
\end{proof}

\section{Almost-minimality and interior regularity}\label{section interior regularity} In this section we gather some simple basic
properties of the almost-minimizers introduced in Definition \ref{DFAL}, show that minimizers in \eqref{variational problem} are almost-minimizers, and then check the interior regularity theory from \cite{caputoguillen} applies in our case. Let us recall that given $K\in\K(n,s,\l)$, open sets $\Om$ and $A$ with
  \begin{equation}
    \label{condition Omega A section}
      I(\Om A,\Om^c)<\infty\,,
  \end{equation}
and $\Lambda\in[0,\infty)$, $r_0\in(0,\infty]$ and $\s\in(-1,1)$, we say that $E\subset\Om$ is {\it $(\Lambda,r_0,\s,K)$-minimizer in $(A,\Omega)$} if
\begin{eqnarray}
    \label{Lambda r0 minimizer section}
  &&I(E A,E^c \Om)+I(E A^c,E^c\Om A)+
\s\,I(E A,\Om^c)
  \\\nonumber
  &\le&
  I(F A,F^c \Om)+
I(F A^c,F^c \Om  A)+\s\,I(F A,\Om^c)+\Lambda\,|E\Delta F|\,,
  \end{eqnarray}
whenever $F\subset\Om$,
$\diam(F\Delta E)<2\,r_0$ and $F\cap A^c=E\cap A^c$. Thanks to \eqref{condition Omega A}, $I(F A,\Om^c)<\infty$ whenever $F\subset\Om$, and in particular the right hand side of \eqref{Lambda r0 minimizer section} is always well definite in $(-\infty,\infty]$. We begin with two simple remarks.

\begin{remark}[Almost-minimality and blow-ups]\label{remark blowups}
  {\rm Let us recall our notation $A^{x,r}=(A-x)/r$ for the blow-up of $A\subset\R^n$ near $x\in\R^n$ at scale $r>0$. It is easily seen that for every $x\in\R^n$ and $r>0$ one has that $E$ is a $(\Lambda,r_0,\s,K)$-minimizer in $(A,\Omega)$ if and only if
  \[
  \mbox{$E^{x,r}$ is a $(r^s\Lambda,r_0/r,\s,r^{n+s}\,K(r\,\cdot))$-minimizer in $(A^{x,r},\Omega^{x,r})$}\,.
  \]
  In particular, should $E^{x,r}$ converge to a limit set $E^*$ as $r\to 0^+$ (for some $x\in A$ fixed), then one expects $E^*$ to be a $(0,\infty,\s,K^*)$-minimizer in $(B_R,H)$ for every $R>0$, with $H=\R^n$ if $x\in A\cap\Om$, and with $H=\{z:z\cdot \nu_\Om(x)<0\}$ if $x\in A\cap\pa\Om$ and $\Om$ is an open set of class $C^1$. Here $K^*$ is defined as in \eqref{homogeneous kernel}.}
\end{remark}

\begin{remark}[Almost-minimality and complement]\label{remark complement}
  {\rm One notices that $E$ is a $(\Lambda,r_0,\s,K)$-minimizer in $(A,\Omega)$ if and only if
  \[
  \mbox{$\Om\cap E^c$ is a $(\Lambda,r_0,-\s,K)$-minimizer in $(A,\Omega)$.}
  \]
This can be easily checked by noticing that, for any set~$E\subset\Omega$,
\begin{eqnarray*}
I(\Omega E^c A,(\Omega E^c)^c\Omega)
&=& I(\Omega E^c A,E)
\\
&=& I(E A,E^c \Omega A)+I(E A^c,E^c\Omega A)\,,
\\
I(\Omega E^c A^c,(\Omega E^c)^c\Omega A)
&=&I(E A,E^c\Omega A^c)\,,
\\
\s\,I (\Omega E^c A,\Omega^c)&=&-\s\,I (E A, \Omega^c) +\s\, I (\Omega A, \Omega^c)\,.
\end{eqnarray*}
  }
\end{remark}

Let us recall the definition of (nonlocal) relative perimeter of $E$ in an open set $A$,
\[
P(E;A)=I(EA,E^cA)+I(EA,E^cA^c)+I(EA^c,E^cA)\,.
\]

\begin{proposition}
  If $K\in\K(n,s,\l)$ and $E$ is a $(\Lambda,r_0,\s,K)$-minimizer in $(A,\Omega)$ and $x_0$ and ${{\varrho}}_0$ are such that $B_{2\,{{\varrho}}_0}(x_0)\cc\Om\cap A$ with ${{\varrho}}_0\le r_0$, then
  \begin{equation}
    \label{almost min caputo}
    P(E;B_{{{\varrho}}_0}(x_0))\le P(F;B_{{{\varrho}}_0}(x_0))+C\,\frac{|E\Delta F|}{{{\varrho}}_0^s}
  \end{equation}
  for every set $F$ such that $E\Delta F\cc B_{{{\varrho}}_0}(x_0)$, where $C$ depends on $\Lambda$, $\l$, $n$ and $s$.
\end{proposition}

\begin{proof}
  Since ${{\varrho}}_0\le r_0$ we can plug any $F$ such that $E\Delta F\cc B_{{{\varrho}}_0}(x_0)$ into \eqref{Lambda r0 minimizer section}, and then deduce
  \begin{eqnarray*}
    I(E A,E^c \Om)+I(E A^c,E^c\Om A)&\le&I(F A,F^c \Om)+I(F A^c,F^c \Om  A)
  \\
  &&+|I(F A,\Om^c)-I(E A,\Om^c)|+\Lambda\,|E\Delta F|\,,
  \end{eqnarray*}
  where $K\in\K(n,s,\l)$ gives
  \begin{eqnarray*}
    |I(F A,\Om^c)-I(E A,\Om^c)|&\le&\l\int_{\Om^c}dy\,\int_{(E\Delta F)\cap B_{{{\varrho}}_0}(x_0)}\frac{dx}{|x-y|^{n+s}}
    \\
    &\le&\l\,|E\Delta F|\,\int_{B_{2{{\varrho}}_0}(x_0)^c}\frac{dy}{\dist(y,B_{{{\varrho}}_0}(x_0))^{n+s}}
    \\
    &\le&\frac{\l}{{{\varrho}}_0^s}\,n\om_n\,\int_2^\infty\frac{t^{n-1}\,dt}{(t-1)^{n+s}}|E\Delta F|
    \le C\,\frac{|E\Delta F|}{{{\varrho}}_0^s}\,.
  \end{eqnarray*}
  We thus have
  \begin{eqnarray}\label{bri}
    I(E A,E^c \Om)+I(E A^c,E^c\Om A)\le I(F A,F^c \Om)+I(F A^c,F^c \Om  A)
  +C\,\frac{|E\Delta F|}{{{\varrho}}_0^s}\,.
  \end{eqnarray}
  Let us now set $W=B_{{{\varrho}}_0}(x_0)$ for the sake of brevity. Since $W\cc\Om\cap A$ we have
  \begin{eqnarray*}
    I(E A,E^c \Om)+I(E A^c,E^c\Om A)&=&I(E W,E^cW)+I(E W,E^cW^c\Om)
    \\
    &&+I(E A W^c,E^cW)+I(E A W^c,E^cW^c\Om)
    \\
    &&+I(EA^c,E^cW)+  I(EA^c,E^c\Om A W^c)
  \end{eqnarray*}
  where $E\Delta F\cc W\cc A$ implies that by replacing $E$ with $F$ we leave unchanged both the fourth and sixth interaction terms. We denote by $\k$ their sum, so that $\k(E)=\k(F)$, and rewrite the above identity as
  \begin{eqnarray*}
    I(E A,E^c \Om)+I(E A^c,E^c\Om A)&=&I(E W,E^cW)+I(E W,E^cW^c\Om)
    \\
    &&+I(E A W^c,E^cW)+I(EA^c,E^cW)+ \k
    \\
    &=&I(E W,E^cW)+I(E W,E^cW^c\Om)+I(E W^c,E^cW)+\k
    \\
    &=&P(E; W)-I(E W,E^c\Om^c)+\k\,.
  \end{eqnarray*}
  Hence \eqref{bri} is equivalent to
  \begin{eqnarray}\label{bri2}
    P(E;W)\le P(F;W)+I(E W,E^c\Om^c)-I(F W,F^c\Om^c)
  +C\,\frac{|E\Delta F|}{{{\varrho}}_0^s}\,.
  \end{eqnarray}
  But since $E^c\cap\Om^c=F^c\cap\Om^c$, by arguing as before we find
  \[
  |I(E W,E^c\Om^c)-I(F W,F^c\Om^c)|\le\int_{\Om^c}\,dy\int_{(E\Delta F)\cap W}K(x-y)\,dx\le C\,\frac{|E\Delta F|}{{{\varrho}}_0^s}\,,
  \]
  and \eqref{almost min caputo} is proved.
\end{proof}

\begin{corollary}
  If $E$ is a $(\Lambda,r_0,\s,K_s)$-minimizer in $(A,\Omega)$, then there exists a relatively closed subset $\S$ of $\Om\cap\pa E$ such that $\Om\cap\pa E\setminus\S$ is a $C^{1,\a}$-hypersurface for some $\a\in(0,1)$ and $\S$ has Hausdorff dimension at most $n-3$. In particular, $\S$ is empty if $n=2$.
\end{corollary}

\begin{proof}
  The validity of \eqref{almost min caputo} allows one to apply the main result of \cite{caputoguillen} and the deduce the above assertion with the Hausdorff dimension of $\S$ bounded by $n-2$. The improvement on the dimensional bound for $\S$ is obtained by exploiting \cite{savinvaldinoci}.
\end{proof}

We now show that minimizers in \eqref{variational problem} are almost-minimizers.

\begin{proposition}\label{corollary minimi sono lambda minimi}
  If $E$ is a minimizer in \eqref{variational problem}, then $E$ is a $(\Lambda,r_0,\s,K)$-minimizer in $(\R^n,\Om)$ for values of $r_0$ and $\Lambda$ depending on $E$ and $\|g\|_{L^\infty(\Om)}$ only.
\end{proposition}

\begin{proof}
   Let us fix two points $x_0\ne y_0\in\Om\cap\pa E$ so that for some ${{\varrho}}_0>0$ we have
   \[
   |E\cap B_{{{\varrho}}_0}(x_0)|>0\,,\qquad |E\cap B_{{{\varrho}}_0}(y_0)|>0\,,\qquad |x_0-y_0|>4\,{{\varrho}}_0\,,\qquad B_{{{\varrho}}_0}(x_0)\cup B_{{{\varrho}}_0}(y_0)\cc\Om\,.
   \]
   Then there exists $T\in C^\infty_c(B_{{{\varrho}}_0}(x_0);\R^n)$ and $S\in C^\infty_c(B_{{{\varrho}}_0}(y_0);\R^n)$ such that
   \[
   \int_E\Div\,T=\int_E\Div\,S=1\,,
   \]
   see, e.g. \cite[Lemma 3.5]{colombomaggi}. Let us now pick $F\subset\Om$ with $\diam(F\Delta E)<2\,r_0$. If $r_0$ is small enough with respect to ${{\varrho}}_0$, then we either have $\dist(F,B_{{{\varrho}}_0}(x_0))>0$ or $\dist(F,B_{{{\varrho}}_0}(y_0))>0$. Without loss of generality, we may assume to be in the first case. Now let $f_t(x)=x+t\,T(x)$ and define
   \[
   F_t=\Big(f_t(E)\cap B_{{{\varrho}}_0}(x_0)\Big)\cup\Big(F\setminus B_{{{\varrho}}_0}(x_0)\Big)=f_t(F)
   \]
   for $|t|<\e_0$ and $\e_0$ small enough to ensure that $\{f_t\}_{|t|<\e_0}$ is a family of smooth diffeomorphisms with $\spt(f_t-\Id)\cc B_{{{\varrho}}_0}(x_0)$ for every $|t|<\e_0$. If we set $\vphi(t)=|F_t|$, then
   \[
   \vphi'(0)=\int_E\,\Div\,T=1\,,
   \]
   so that, up to decreasing the value of $\e_0$, $\vphi$ is strictly increasing on $(-\e_0,\e_0)$, with range $(-v_0,v_0)$ for some $v_0>0$. Notice that the size of $v_0$ only depends on $E$ through the choice of $x_0$ and of the vector field $T$. Thus, up to decreasing the value of $r_0$ depending on $E$, we find that $||F|-|E||<\om_n\,r_0^n<v_0$, and thus that there exists $t_*=t_*(F)$ such that
   \[
   |F_{t_*}|=|E|\qquad |t_*|\le C\,\big||F|-|E|\big|
   \]
   for a constant $C=C(E)$. By minimality of $E$ we have
   \[
   I(E,E^c\Om)+\s\,I(E,\Om^c)+\int_E g\le I(F_{t_*},F_{t_*}^c\Om)+\s\,I(F_{t_*},\Om^c)+\int_{F_{t_*}}g\,.
   \]
   Now, since for some $C=C(E)$ we have $|Jf_t(x)-1|\le C\,|t|$ and $|\nabla f_t|\le C$ on $\R^n$ for every $|t|<\e_0$, by the area formula we find
   \begin{eqnarray*}
   &&\big|I(F_t,F_t^c\Om)-I(F,F^c\Om)\big| \le C\,|t|\,I(F,F^c\Om)\,,
   \\
   &&\big|I(F_t,\Om^c)-I(F,\Om^c)\big|\le C\,|t|\,I(F,\Om^c)\,,
   \\
   &&\Big|\int_{F_t}g-\int_Eg\Big|\le \Big|\int_{F_t}g-\int_Fg\Big|+\|g\|_{L^\infty(\Om)}\,|E\Delta F|\le
   C\,|t|+\|g\|_{L^\infty(\Om)}\, |E\Delta F|\,,
   \end{eqnarray*}
   whenever $|t|<\e_0$. By exploiting these facts with $t=t_*$ and taking into account $ |t_*|\le C\,\big||F|-|E|\big|$, we conclude that
   \[
   I(E,E^c\Om)+\s\,I(E,\Om^c)+\le I(F,F^c\Om)+\s\,I(F,\Om^c)+C\,\big|E\Delta F\big|\,,
   \]
   where $\Lambda=\Lambda(E,\|g\|_{L^\infty(\Om)})$.
\end{proof}

\section{Density estimates at the boundary}\label{section boundary regularity}
We now discuss the proof of Theorem \ref{thm density estimate intro}. We shall actually prove a more general result, involving the following notion of uniformly $C^1$ domain.

\begin{definition}\label{def rho Omega}
  {\rm If $\eta>0$, $A$ is an open set, $\Om$ is an open set in $\R^n$ with boundary of class $C^1$ in $A$, and $H_p$ denotes the affine tangent half-space to $\Om$ at $p\in\pa\Om$, then we define
  \[
  \varrho_A(\eta,\Om)
  \]
  as the supremum of all $\varrho>0$ such that for every $p\in A\cap\pa\Om$  there exists a $C^1$-diffeomorphisms $T_p:\R^n\to\R^n$ with
  \begin{gather}
  \label{Tp Brho}
  T_p(B_{\varrho}(p))=B_{\varrho}(p)\,,
  \\\label{Tp BrhoOmega}
  T(B_{\varrho}(p)\cap\Om)=B_{\varrho}(p)\cap H_p\,,
  \\\label{Tp C0}
  \|T_p -\Id\|_{C^0(\R^n)}+\|T_p^{-1}-\Id\|_{C^0(\R^n)}\le \eta\,{{\varrho}}\,,
  \\\label{Tp C1}
  \|\nabla T_p -\Id\|_{C^0(\R^n)}+\|(\nabla T_p)^{-1}-\Id\|_{C^0(\R^n)}\le \eta\,.
  \end{gather}}
\end{definition}

\begin{remark}\label{remark stable blowup}
  {\rm If $\Om$ is a {\it bounded} open set with $C^1$-boundary, then $\varrho_{\R^n}(\eta,\Om)>0$; but, of course, one can have $\varrho_{\R^n}(\eta,\Om)>0$ even if $\Om$ is unbounded (for example, if $\varrho_{\R^n}(\eta,H)=\infty$ if $H$ is a half-space). We also notice that for every $x_0\in\R^n$ and $r>0$ one has
  \begin{equation}
    \label{scaling varrho}
      \varrho_A(\eta,\Om)=r\,\varrho_{A^{x_0,r}}(\eta,\Om^{x_0,r})\,.
  \end{equation}
  Indeed, given a set of maps $\{T_p\}_{p\in\pa\Om}$ associated to some $\varrho<\varrho_A(\eta,\Om)$ one can use the maps $\{S_q\}_{q\in\Om^{x_0,r}}$ defined by
  \[
  p=x_0+r\,q\,,\qquad S_q(y)=\frac{T_p(x_0+r\,y)-x_0}r\,,
  \]
  to show that $\varrho/r<\varrho_{A^{x_0,r}}(\eta,\Om^{x_0,r})$. In particular, $\varrho_A(\eta,\Om)\le\varrho_{A^{x_0,r}}(\eta,\Om^{x_0,r})$ for every $r\in(0,1)$, that is, the positivity of $\varrho_{A^{x_0,r}}(\eta,\Om^{x_0,r})$ is stable under blow-ups of $\Om$. Identity \eqref{scaling varrho} is needed to obtain density estimates that are stable under blow-up limits.}
\end{remark}

With Definition \ref{def rho Omega}, we can formulate the following improved version of Theorem \ref{thm density estimate intro}. Notice that the assumption of $\Omega$ being a bounded open set with $C^1$-boundary or an half-space is replaced here
by the requirement that ${{\varrho}}_A(\eta,\Om)>0$ for every $\eta>0$.

\begin{theorem}
  [Density estimates]\label{thm density estimate full} Let $n\ge 2$, $s\in(0,1)$, $\s\in(-1,1)$, $\Lambda\ge0$ and $K=K_s^\e$ for some $\e>0$.
  If $A$ is an open set and $\Om$ is an open set with $C^1$-boundary in $A$ such that $\varrho_A(\eta,\Om)>0$ for every $\eta>0$, then there exist positive constants $C_0$  (depending on $n$, $s$, $\s$ and $\Lambda$), $c_*$ (depending on $n$ and $s$) and $\eta_1$ (depending on $n$, $s$ and $\s$) with the following property: for every $(\Lambda,r_0,\s,K_s^\e)$-minimizer $E$ in $(A,\Omega)$, one has
  \begin{equation}
    \label{upper perimeter estimate}
      I_s^\e(E B_r(x),(E B_r(x))^c)\le C_0\,r^{n-s}\,,
  \end{equation}
  whenever $B_r(x)\subset A$ and $r<\min\{r_0,c_*\,\varrho_A(\eta_1,\Om),c_*\,\e\}$, and, moreover,
  \begin{equation}
    \label{volume density estimates}
      \frac1{C_0}\le \frac{|E\cap B_r(x)|}{r^n}\le 1-\frac1{C_0}\,,
  \end{equation}
  whenever $B_r(x)\subset A$, $r<\min\{r_0,\,c_*\varrho_A(\eta_1,\Om),\,c_*\,\e\}$, and $x\in\overline{\Om\cap\pa E}$.
\end{theorem}

We now turn to the proof of Theorem \ref{thm density estimate full}. A key tool is a geometric inequality, stated in Lemma \ref{lemma bordo} below, which can be introduced by the following considerations. A crucial role in the study of local capillarity problems is played by the geometric remark that, if ${\rm Per}$ denotes the classical (local) perimeter, then
\begin{equation}
  \label{minimality halfspace local}
  {\rm Per}(Z;H)\ge {\rm Per}(Z;\pa H)\,,
\end{equation}
whenever $Z\subset H$ is of finite perimeter and finite volume (this is a consequence of the divergence theorem; see, for example, \cite[Proposition 19.22]{maggiBOOK}). An analogous inequality to \eqref{minimality halfspace local} holds for fractional perimeters too: if $H$ is a half-space in $\R^n$ and $Z$ is a bounded subset of $H$, then
\begin{equation}
  \label{minimality halfspace}
  I_s(Z,Z^c H)\ge I_s(Z,H^c)\,.
\end{equation}
Indeed, let $R>0$ be such that $Z\subset H\cap B_R$. If we set $J=H^c$ and $Y=J\cup Z$, then $J$ is a half-space and $Y\setminus B_R=J\setminus B_R$, so that, by \cite[Corollary 5.3(b)]{caffaroquesavin} (see also \cite[Proposition 17]{ambrosiodephilippismartinazzi}),
\begin{eqnarray}\label{minimality halfspace1}
I_s(Y B_R,Y^c)+I_s(Y B_R^c,Y^c B_R)\ge I_s(J B_R,J^c)+I_s(J B_R^c,J^c B_R)\,.
\end{eqnarray}
Since $Y^c=Z^c\cap H$ and $Z\subset H\cap B_R$, one finds
\begin{eqnarray*}
  &&I_s(Y B_R,Y^c)=I_s(Z,Z^c H)+I_s(H^c B_R,Z^c H)\,,
  \\
  &&I_s(Y B_R^c,Y^c B_R)=I_s(H^c B_R^c,H Z^c B_R)\,,
\end{eqnarray*}
which, combined with \eqref{minimality halfspace1}, gives
\begin{eqnarray*}
  I_s(Z,Z^c  H)&\ge& I_s(H^c  B_R,H)-I_s(H^c  B_R,Z^c  H)
  \\
  &&+I_s(H^c  B_R^c,H  B_R)-I_s(H^c  B_R^c,H  Z^c  B_R)
  \\
  &=&
  I_s(H^c  B_R,Z)+I_s(H^c  B_R^c,Z)=I_s(Z,H^c)\,.
\end{eqnarray*}
(This argument actually shows that \eqref{minimality halfspace} is equivalent to \eqref{minimality halfspace1}.)
We now want to generalize \eqref{minimality halfspace} to the case when an open set $\Om$ takes the place of the half-space $H$. The idea is that on sets of sufficiently small diameter, if the boundary $\Om$ is regular enough to be locally close to a half-space at each of its boundary points, then an inequality like \eqref{minimality halfspace} should hold true with some error terms.

\begin{lemma}\label{lemma bordo}
  Given $n\ge 2$, $s\in(0,1)$, and $\e>0$ there exist positive constants $C_\star$ and $\eta_0$ (depending on $n$ and $s$, and with $C_\star\,\eta_0<1$) with the following property. If $A$ is an open set, $\Om$ is an open set with $C^1$-boundary in $A$, $\eta\in(0,\eta_0)$,
  \begin{equation}
    \label{rstar}
      r_\star:=\min\Big\{\frac{\varrho_A(\eta,\Om)}{4\,C_\star},\frac{\e}{2\,C_\star}\Big\}
  \end{equation}
  and
  \begin{equation}
    \label{nested G}
      G\subset\Om\cap B_{r_\star}(x)\qquad \mbox{for some $x\in\R^n$}
  \end{equation}
  then
  \begin{equation}
    \label{geometrica bordo}
    I_s^\e(G,G^c \Om)\ge(1-C_\star\,\eta)\,I_s^\e(G,\Om^c)-\frac{C_\star}{r_\star^s}\,|G|\,,
  \end{equation}
\end{lemma}

\begin{proof} Let us fix $\eta\in(0,\eta_0)$, assume without loss of generality that $\varrho_A(\eta,\Om)>0$, define $r_\star$ by \eqref{rstar}, and directly consider the case $|G|>0$. The idea is that when $B_{r_\star}(x)$ is sufficiently close to $\pa \Om$, then one can first ``flatten'' the boundary and then exploit the local minimality of half-spaces expressed in \eqref{minimality halfspace} in order to obtain \eqref{geometrica bordo}. If, instead,
$B_{r_\star}(x)$ is away from $\pa\Om$ then \eqref{geometrica bordo} follows by the isoperimetric inequality (for the fractional perimeter $P_s$).

\medskip

\noindent {\it Step one}: We prove that if, in addition to \eqref{nested G}, we have $B_{C_\star r_\star}(x)\subset\Omega$, then
\begin{equation}\label{st1}
I_s^\e(G,G^c\Om)\ge I_s^\e(G,\Om^c)\,.
\end{equation}
First we notice that, trivially,
\begin{equation}
  \label{st000}
|z-y|\ge\frac{C_\star-1}{C_\star}|z-x|\qquad\forall y\in B_{r_\star}(x)\,,z\in\ B_{C_\star\,r_\star}(x)^c\,.
\end{equation}
(We definitely assume that $C_\star>1$.)
By assumption we have $G\subset B_{r_\star}(x)$ and $\Om^c\subset B_{C_\star\,r_\star}(x)^c$, so that \eqref{st000} and $K_s^\e=1_{B_\e}\,K_s$ give us
\begin{eqnarray}\nonumber
I_s^\e(G,\Omega^c)&\le&I_s^\e(G,B_{C_\star r_\star}(x)^c)\le\left(\frac{C_\star}{C_\star-1}\right)^{n+s}
\int_{G}dy\int_{B_{C_\star r_\star}(x)^c}\frac{dz}{|z-x|^{n+s}}
\\\label{st2}
&\le&n\,\om_n \left( \frac{C_\star}{C_\star-1}\right)^{n+s}\,|G|\,\int_{C_\star r_\star}^{+\infty} \varrho^{-1-s}\,d\varrho=
\frac{n\,\om_n}{s}\left(\frac{C_\star}{C_\star-1}\right)^{n+s}\,\frac{|G|}{(C_\star r)^s}\hspace{0.9cm}
\end{eqnarray}
where~$\omega_n$ is the volume of the unit ball. At the same time, by the fractional isoperimetric inequality (see \cite{frankliebseiringer,caffarellivaldinoci}) we have that
\begin{equation}
  \label{fractional isoperimetric inequality}
  P_s(G)\ge \frac{P_s(B_1)}{\om_n^{(n-s)/n}}\,|G|^{(n-s)/n}\,.
\end{equation}
Since $(C_\star+1)\,r_\star\le 2\,C_\star r_\star<\e$ and $G\subset B_{r_\star}(x)$ we have that
\[
|y-z|\le\e\qquad\forall y\in G \,, z\in B_{C_\star r_\star}(x)\,,
\]
so that by $K_s^\e=1_{B_\e}K_s$ and by $B_{C_\star r_\star}(x)\subset\Om$ we find
\[
I_s^\e(G,G^c\Om)\ge I_s^\e(G,G^cB_{C_\star r_\star}(x))\ge I_s(G,G^cB_{C_\star r_\star}(x))=
\Big(P_s(G)-I_s(G;B_{C_\star r_\star}(x)^c)\Big)\,.
\]
Hence, by \eqref{fractional isoperimetric inequality} and \eqref{st2}, we have
\begin{eqnarray*}
\frac{I_s^\e(G,G^c\Om)}{I_s^\e(G,\Om^c)}&\ge& \frac{ \frac{P_s(B_1)}{\om_n^{(n-s)/n}}\,|G|^{(n-s)/n}
-\frac{n\,\om_n}{s}\left(\frac{C_\star}{C_\star-1}
\right)^{n+s}\,\frac{|G|}{(C_\star r)^s}  }{\frac{n\,\om_n}{s}\left(\frac{C_\star}{C_\star-1}
\right)^{n+s}\,\frac{|G|}{(C_\star r)^s} }
\\
&=&
\frac{s}{n}\,P_s(B_1)\Big(\frac{C_\star-1}{C_\star}\Big)^{n+s}\,C_\star^s\,\Big(\frac{|B_{r_\star}(x)|}{|G|}\Big)^{s/n}-1
\\
&\ge& \frac{s}{n}\,P_s(B_1)\Big(\frac{C_\star-1}{C_\star}\Big)^{n+s}\,C_\star^s\,-1 \ge 1\,,
\end{eqnarray*}
where the last inequality holds provided $C_\star$ is large enough depending on $n$ and $s$.

\medskip

\noindent {\it Step two}: We now complete the proof of the lemma. We first notice that
\begin{equation}\label{ABE}
|K_s(\zeta_1)-K_s(\zeta_1)|\le C(n,s)\,\frac{K_s(\zeta_1)}{|\zeta_1|}\,|\zeta_1-\zeta_2|
\end{equation}
whenever $|\zeta_1-\zeta_2|\le|\zeta_1|/2$.
Indeed, if~$t\in[0,1]$, then
$$ |t\zeta_2+(1-t)\zeta_1|\ge |\zeta_1|-|\zeta_2-\zeta_1|\ge
\frac{|\zeta_1|}{2}.$$
and thus
\begin{eqnarray*}
|K_s(\zeta_2)-K_s(\zeta_1)|&\le&\sup_{t\in[0,1]}
|\nabla K_s(t\zeta_2+(1-t)\zeta_1)|\,|\zeta_2-\zeta_1|
\\
&\le&
\sup_{t\in(0,1)}\frac{|\zeta_2-\zeta_1|}{|t\zeta_2+(1-t)\zeta_1|^{n+s+1}}
\le C(n,s)\,\frac{K_s(\zeta_1)}{|\zeta_1|}\,|\zeta_2-\zeta_1|\,.
\end{eqnarray*}
This proves~\eqref{ABE}, which we
are going to use now in the proof of \eqref{geometrica bordo}.

Given step one, we may directly assume that there exists $p\in B_{C_\star r_\star}(x)\cap\pa\Om$ (as well as that $B_{r_\star}(x)\cap\Om\ne\varnothing$, otherwise \eqref{nested G} would give $G=\varnothing$). The existence of $p$ gives
\begin{equation}\label{st5}
B_{2\,r_\star}(x)\subset B_{\varrho}(p)\,,\qquad\mbox{for some $\varrho<\varrho(\eta,\Omega)$}\,.
\end{equation}
Indeed, if~$q\in B_{2\,r_\star}(x)$ and we pick $C_\star>2$, then we find
\[
|q-p|\le |q-x|+|x-p|\le 2\,r_\star+C_\star r_\star<2 C_\star r_\star<\varrho(\eta,\Omega)
\]
by definition of $r_\star$. By definition of $\varrho_A(\eta,\Omega)$ there exists a $C^1$-diffeomorphisms $T_p:\R^n\to\R^n$ such that \eqref{Tp Brho}--\eqref{Tp C1} hold with $\varrho$ as in \eqref{st5}. In particular \eqref{Tp C1}
gives that
\[
\big| \big( T_p(z)-T_p(y)\big) -(z-y)\big|
\le\eta\,|z-y|\,,\qquad\forall z,\;y\in\R^n\,,
\]
so that, provided $\eta_0<1/2$, in view of~\eqref{ABE},
\begin{eqnarray*}
  \Big|K_s\big(T_p(z)-T_p(y)\big)-K_s(z-y)\Big|\le C(n,s)\,K_s(z-y)\,\eta\,,\qquad\mbox{if $|z-y|<\e$}\,.
\end{eqnarray*}
At the same time \eqref{Tp C1} implies
\begin{eqnarray*}
  \Big|JT_p(z)\,JT_p(y)-1\Big|\le C(n)\,\eta\qquad\forall z,y\in\R^n\,,
\end{eqnarray*}
so that in conclusion, for every $z,y\in\R$ we have
\[
\Big|JT_p(z)\,JT_p(y)\,K_s\big(T_p(z)-T_p(y)\big)-K_s(z-y)\Big|\le C_\star\,\eta\,K_s(z-y)\,.
\]
By the area formula, for every pair of disjoint sets $A_1, A_2\subset\R^n$ one has that $I_s(A_1,A_2)$ if finite if and only if $I_s(T_p(A_1),T_p(A_2))$ is finite, with
\begin{equation}
  \label{circa}
(1-C_\star\,\eta)I_s(A_1,A_2)\le I_s(T_p(A_1),T_p(A_2))\le (1+C_\star\,\eta)\,I_s(A_1,A_2)\,.
\end{equation}
We are now in the position to conclude our argument. By $G\subset B_{r_\star}(x)\cap\Om$ and \eqref{st5} we find $T_p(G)\subset H\cap B_{\varrho}(p)$, where $H=H_p$ is the affine tangent half-space to $\Om$ at $p$ (see Definition \ref{def rho Omega}). Thus we can apply \eqref{minimality halfspace} to $Z=T_p(G)$ and find
\begin{equation}
  \label{delicata}
  I_s(T_p(G),T_p(G)^c  H)\ge I_s(T_p(G),H^c)\,,
\end{equation}
which is equivalently written (by using $T_p(G)^c\cap B_{\varrho}(p)^c=B_{\varrho}(p)^c$) as
\begin{eqnarray*}
&&I_s(T_p(G),T_p(G)^c  H  B_{\varrho}(p))+
I_s(T_p(G),H  B_{\varrho}(p)^c)
\\
&&\ge
I_s(T_p(G),H^c  B_{\varrho}(p))+I_s(T_p(G),H^c  B_{\varrho}(p)^c)\,.
\end{eqnarray*}
Since $H\cap B_{\varrho}(p)=T_p(\Om\cap B_{\varrho}(p))$ and $H^c\cap B_{\varrho}(p)=T_p(\Om^c\cap B_{\varrho}(p))$, by exploiting  \eqref{circa} we obtain
\begin{eqnarray}\label{mt1}
&&(1+C_\star\,\eta)\,I_s(G,G^c  \Om  B_{\varrho}(p))+
I_s(T_p(G),H  B_{\varrho}(p)^c)
\\\nonumber
&\ge&
(1-C_\star\,\eta)\,I_s(G,\Om^c  B_{\varrho}(p))+I_s(T_p(G),H^c  B_{\varrho}(p)^c)\,.
\end{eqnarray}
Again by \eqref{circa} one finds
\begin{eqnarray*}
  &&|I_s(T_p(G),H  B_{\varrho}(p)^c)-I_s(G,\Om  B_{\varrho}(p)^c)|
  \\
  &\le&|I_s(T_p(G),T_p(\Om  B_{\varrho}(p)^c))-I_s(G,\Om  B_{\varrho}(p)^c)|
  \\
  &&+|I_s(T_p(G),T_p(\Om  B_{\varrho}(p)^c))-I_s(T_p(G),H  B_{\varrho}(p)^c)|
  \\
  &\le& C_\star\eta I_s(G,\Om  B_{\varrho}(p)^c)+I_s(T_p(G),B_{\varrho}(p)^c)
  \\
  &\le& C_\star\eta I_s(G,\Om  B_{\varrho}(p)^c)+(1+C_\star\,\eta)\,I_s(G,B_{\varrho}(p)^c)\,,
\end{eqnarray*}
so that \eqref{mt1} gives
\begin{eqnarray}\label{mt2}
&&(1+C_\star\,\eta)\,\big(I_s(G,G^c  \Om)+I_s(G,B_{\varrho}(p)^c)\big)
\\\nonumber
&\ge&
(1-C_\star\,\eta)\,I_s(G,\Om^c  B_{\varrho}(p))+I_s(T_p(G),H^c  B_{\varrho}(p)^c)\,.
\end{eqnarray}
Similarly,  by using \eqref{circa} one more time,
\begin{eqnarray*}
  &&|I_s(T_p(G),H^c  B_{\varrho}(p)^c)-I_s(G,\Om^c  B_{\varrho}(p)^c)|
  \\
  &\le&|I_s(T_p(G),T_p(\Om^c  B_{\varrho}(p)^c))-I_s(G,\Om^c  B_{\varrho}(p)^c)|
  \\
  &&+|I_s(T_p(G),T_p(\Om^c  B_{\varrho}(p)^c))-I_s(T_p(G),H^c  B_{\varrho}(p)^c)|
  \\
  &\le& C_\star\,\eta I_s(G,\Om^c  B_{\varrho}(p)^c)+I_s(T_p(G),B_{\varrho}(p)^c)
  \\
  &\le& C_\star\,\eta I_s(G,\Om^c  B_{\varrho}(p)^c)+(1+C_\star\,\eta)\,I_s(G,B_{\varrho}(p)^c)\,,
\end{eqnarray*}
which plugged into \eqref{mt2} gives, as $C_\star\eta_0<1$,
\begin{eqnarray}\label{mt3}
(1+C_\star\,\eta)\,I_s(G,G^c  \Om)\ge (1-C_\star\,\eta)\,I_s(G,\Om^c)-4\,I_s(G,B_{\varrho}(p)^c)\,.
\end{eqnarray}
Since $G\subset B_{r_\star}(x)$ with $B_{2\,r_\star}(x)\subset B_{\varrho}(p)$, recalling~\eqref{kernelclass 0}
we have
\begin{eqnarray*}
I_s(G,B_{\varrho}(p)^c)\le \,\int_G\,dz\int_{B_{2\,r_\star}(x)^c}\frac{dy}{|z-y|^{n+s}}=
  \frac{n\om_n}{s\,2^s\,r_\star^s}\,|G|\,,
\end{eqnarray*}
and thus \eqref{mt3} implies
\begin{eqnarray}\label{thanks}
(1+C_\star\,\eta)\,I_s(G,G^c  \Om)&\ge&(1-C_\star\,\eta)\,I_s^\e(G,\Om^c)-4\,I_s(G,B_{\varrho}(p)^c)
\\
&\ge&(1-C_\star\,\eta)\,I_s^\e(G,\Om^c)-C(n,s)\frac{|G|}{r_\star^s}\,.
\end{eqnarray}
Now
\[
I_s(G,G^c  \Om)=I_s^\e(G,G^c  \Om)+\int_G\,dz\int_{B_\e(z)\cap G^c\cap\Om}\frac{dy}{|z-y|^{n+s}}\le
I_s^\e(G,G^c  \Om)+C(n,s)\frac{|G|}{\e^s}
\]
and so we deduce \eqref{geometrica bordo} from \eqref{thanks} and $r_\star\le \e/C(n,s)$.
\end{proof}

We are now ready for the proof of Theorem \ref{thm density estimate full}.

\begin{proof}[Proof of Theorem~\ref{thm density estimate full}]
  Let $E$ be a $(\Lambda,r_0,\s,K_\s^\e)$-minimizer in $(A,\Om)$, so that
    \begin{eqnarray}
    \label{Lambda r0 minimizer proof}
  &&I^\e_s(E A,E^c \Om)+I^\e_s(E A^c,E^c\Om A)+
\s\,I^\e_s(E A,\Om^c)
  \\\nonumber
  &\le&
  I^\e_s(F A,F^c \Om)+
I^\e_s(F A^c,F^c \Om  A)+\s\,I^\e_s(F A,\Om^c)+\Lambda\,|E\Delta F|\,,
  \end{eqnarray}
  whenever $F\subset\Om$ with $\diam(F\Delta E)<2\,r_0$ and $F\cap A^c=E\cap A^c$.

 Let us fix $B_r(x)\subset A$ with $r<r_0$ and test \eqref{Lambda r0 minimizer proof} with $F=E\cap B_r(x)^c$. Since $F\cap A=E\cap B_r(x)^c\cap A$ and $F^c=E^c\cup(E\cap B_r(x))$ one has
 $$  I^\e_s(F  A,F^c \Om)-I^\e_s(E  A,E^c \Om)
=-I^\e_s(E  B_r(x),E^c \Om)
+I^\e_s(E  B_r(x)^c  A,E  B_r(x))\,.$$
  Similarly, by $F\cap A^c=E\cap A^c$,
  \begin{eqnarray*}
    I^\e_s(F  A^c,F^c \Om  A)-I^\e_s(E  A^c,E^c \Om  A)&=&I^\e_s(E  A^c,E  B_r(x))\,,
    \\
    I^\e_s(F  A,\Om^c)-I^\e_s(E  A,\Om^c)&=&-I^\e_s(E  B_r(x),\Om^c)\,,
  \end{eqnarray*}
  so that \eqref{Lambda r0 minimizer proof} gives
  \begin{equation}\label{density 1}\begin{split}
&  I^\e_s(E  B_r(x),E^c \Om)+\s\,I^\e_s(E  B_r(x),\Om^c)\\
&\qquad\le I^\e_s(E  B_r(x)^c  A,E  B_r(x))
+I^\e_s(E  A^c,E  B_r(x))+\Lambda\,u(r)\,,\end{split}\end{equation}
  provided $u(r)=|E\cap B_r(x)|$. By $A^c\subset B_r(x)^c$ one finds
  \[
  I^\e_s(E  B_r(x)^c  A,E  B_r(x))+I^\e_s(E  A^c,E  B_r(x))\le 2\,I^\e_s(E  B_r(x),E  B_r(x)^c)\,,
  \]
  so that, by adding $I^\e_s(E  B_r(x),E  B_r(x)^c)$ to both sides of \eqref{density 1}, one gets
  \begin{eqnarray}\label{density 2}
  &&I^\e_s(E  B_r(x),(E^c \Om)\cup(E  B_r(x)^c))+\s\,I^\e_s(E  B_r(x),\Om^c)
  \\\nonumber
  &\le&
  3\,I^\e_s(E  B_r(x),E  B_r(x)^c)+\Lambda\,u(r)\,.
  \end{eqnarray}
  Now let $C_\star$ and $\eta_0$ be as in Lemma \ref{lemma bordo} and fix $\eta_1\in(0,\eta_0)$ depending on $n$, $s$,  and $\s$ so that
  \[
  (1-C_\star\,\eta_1)^2-|\s|\ge \eta_1\,.
  \]
  We are going to apply Lemma \ref{lemma bordo} with $\eta=\eta_1$, so that \eqref{rstar} gives
  \[
  r_\star=\min\Big\{\frac{\varrho_A(\eta_1,\Om)}{4\,C_\star},\frac{\e}{2\,C_\star}\Big\}\le c_*\,\min\big\{\varrho_A(\eta_1,\Om),\e\big\}
  \]
  for a constant $c_*$ depending on $n$ and $s$. If we set $G=E\cap B_r(x)$, then $G\subset\Om\cap B_{r_\star}(x)$ provided $r<r_\star$. In particular, by \eqref{geometrica bordo} we find
  \begin{equation}
    \label{ho applicatio geometrica bordo}
        I^\e_s(G,G^c \Om)\ge(1-C_\star\,\eta_1)\,I^\e_s(G,\Om^c)-\frac{C_\star}{r_\star^s}\,|G|\,.
  \end{equation}
  Moreover,
  \[
  G^c\cap\Om=(E^c\cap\Om)\cup(B_r(x)^c\cap\Om)=(E^c\cap\Om)\cup(E\cap B_r(x)^c)\,,
  \]
  so that \eqref{density 2} gives
  $$ 3\,I^\e_s(E  B_r(x),E  B_r(x)^c)+\Lambda\,u(r)\ge
  I^\e_s(G,G^c \Om)+\s\,I^\e_s(G,\Om^c).$$
  By \eqref{geometrica bordo},
  \begin{eqnarray*}
   &&3\,I^\e_s(E  B_r(x),E  B_r(x)^c)+\Lambda\,u(r)
   \\
   &\ge&
   C_\star\,\eta_1\,\,I^\e_s(G,G^c \Om)+(1-C_\star\,\eta_1)I^\e_s(G,G^c \Om)-|\s|\,I^\e_s(G,\Om^c)
   \\
   &\ge&
   C_\star\,\eta_1\,I^\e_s(G,G^c \Om)+[(1-C_\star\,\eta_1)^2-|\s|]\,I^\e_s(G,\Om^c)-(1-C_\star\,\eta_1)\,\frac{C_\star}{r_\star^s}\,u(r)
   \\
   &\ge&
   \eta_1\,\Big(I^\e_s(G,G^c \Om)+I^\e_s(G,\Om^c)\Big)-\frac{C_\star}{r_\star^s}\,u(r)=\eta_1\,P_s^\e(G)-\frac{C_\star}{r_\star^s}\,u(r)\,.
  \end{eqnarray*}
  Summarizing, if $B_r(x)\subset A$ with $r<\min\{r_0,r_\star\}$, then
  \begin{equation}
    \label{density key}
    3\,I^\e_s(E  B_r(x),E  B_r(x)^c)+\Big(\Lambda+\frac{C_\star}{r_\star^s}\Big)\,u(r)\ge \eta_1\,P_s^\e(G)\,,\qquad G=E\cap B_r(x)\,.
  \end{equation}
  Since $I^\e_s(E  B_r(x),E  B_r(x)^c)\le P_s^\e(B_r(x))\le C(n,s)r^{n-s}$ and $u(r) \le\om_n r^n$, we see that \eqref{density key} immediately implies \eqref{upper perimeter estimate}. Next, we apply the fractional isoperimetric inequality \eqref{fractional isoperimetric inequality} to bound from below $P_s^\e(G)$ in \eqref{density key}.  More precisely, we notice that
  \[
  P_s(G)=P_s^\e(G)+\int_G\,dz\int_{G^c\cap B_\e(z)^c}\frac{dy}{|z-y|^{n+s}}\le P_s^\e(G)+C(n,s)\,\frac{|G|}{\e^s}
  \le P_s^\e(G)+C(n,s)\,\frac{|G|}{r_\star^s}
  \]
  so that, up to increasing the value of $C_\star$, \eqref{fractional isoperimetric inequality} gives
  \begin{equation}
    \label{density 3}
      3\,I^\e_s(E  B_r(x),E  B_r(x)^c)+\Big(\Lambda+\frac{C_\star}{r_\star^s}\Big)\,u(r)\ge \frac{P(B_1)}{\om_n^{(n-s)/n}}\,\eta_1\,u(r)^{(n-s)/n}\,.
  \end{equation}
  By exploiting $u(r)\le (\om_n\,r^n)^{s/n} u(r)^{(n-s)/n}$ we find that if
  \begin{equation}
    \label{careful}
      \Big(\Lambda+\frac{C_\star}{r_\star^s}\Big)(\om_n\,r^n)^{s/n}\le \frac{P(B_1)\,\eta_1}{2\,\om_n^{(n-s)/n}}\,,
  \end{equation}
  then \eqref{density 3} implies
  \begin{eqnarray}\label{luis}
    u(r)^{(n-s)/n} \le C(n,s,\Lambda,\s)\,I^\e_s(E  B_r(x),E  B_r(x)^c)\,.
  \end{eqnarray}
  We notice that \eqref{careful} is equivalent to
  \[
  \Big(\frac{r}{r_\star}\Big)^s\le \frac{P(B_1)\,\eta_1}{2\om_n(r_\star^s\Lambda+C_\star)}
  \]
  which, by $r_\star^s\Lambda\le\Lambda$, is in turn implied by
  \[
  r\le c(n,s)\,r_\star
  \]
  and thus by $r\le c_*(n,s)\min\{{{\varrho}}_A(\eta_1,\Om),\e\}$. We have thus proved the validity of \eqref{luis} provided $r\le r_0$ and $r\le c_*(n,s)\min\{{{\varrho}}_A(\eta_1,\Om),\e\}$. Arguing as in \cite[Lemma 4.2]{caffaroquesavin}, we conclude that if $B_r(x)\subset A$, $x\in\Om\cap\pa E$, and $r$ satisfies the above constraints, then $u(r)\ge c_0\,r^n$ for some $c_0=c_0(n,s,\s,\Lambda)$. By Remark \ref{remark complement}, $\Om\cap E^c$ is a $(\Lambda,r_0,-\s)$-minimizer in $(A,\Om)$, and since $\Om\cap\pa(\Om\cap E^c)=\Om\cap\pa E$ one can repeat the above argument with $\Om\cap E^c$ in place of $E$ to find the upper volume density estimate in \eqref{volume density estimates}.
  \end{proof}

\appendix

\section{Closure theorem for almost-minimizers and blow-up limits}\label{appendix blowup limits}

In this appendix we prove a closure theorem for sequences of $(\Lambda,r_0,\s,K)$-minimizers (Theorem \ref{DLVP}). As an application, we then show that blow-up limits exists and are in turn minimizers (Theorem \ref{BL:COMP:TH}). In the following, given an interaction kernel $K\in\K(n,s\l)$, we set
\begin{equation}\label{DEF WOM}
w_F(x):=
1_F(x)\,\int_{F^c}K(x-y)\,dy\qquad F\subset\R^n
\end{equation}
so that
\begin{equation}\label{BEL}
{\mbox{$w_F$ belongs to~$L^1(A)$ if (and only if) $I(AF,F^c)<+\infty$}}.
\end{equation}

\begin{theorem}\label{DLVP}
Let $n\ge2$, $s\in(0,1)$, $\s\in(-1,1)$, $\l\ge 1$, $\Lambda\ge0$, $r_0>0$, $K\in\K(n,s,\l)$ and $A$ be an open set. Consider a sequence $\{E_j\}_{j\in\N}$ of $(\Lambda,r_0,\s,K)$-minimizers in $(A,\Om_j)$, where $\{\Om_j\}_{j\in\N}$ is a sequence of open sets. If there exists an open set $\Om$ with $P(\Om)<\infty$ such that
\begin{equation}
  \label{EjOmj convergono}
  \mbox{$E_j\to E$ and $\Om_j\to\Om$ in $L^1_{{\rm loc}}(A)$}
\end{equation}
and
\begin{equation}\label{DLVPw}
{\mbox{$w_{\Om_j}$ converges to $w_\Om$ weakly in $L^1_{\loc}(A)$}}
\end{equation}
then~$E$ is a $(\Lambda,r_0,\s,K)$-minimizer in~$(A,\Omega)$.

Moreover, in the case when $K=K_s^\e$ and $\Om_j=\Om$ is an open set with $C^1$-boundary such that $\varrho(\eta,\Om)>0$ for every $\eta>0$, one has that:

\medskip

\noindent {\it (i)} if $x_j\in A\cap\ov{\Om\cap\pa E_j}$ and $x_j\to x$ for some $x\in A$, then $x\in\pa E$;

\medskip

\noindent {\it (ii)} if $x\in\ov{\Om\cap\pa E}$ then there exists $x_j\in\pa E_j$ such that $x_j\to x$.
\end{theorem}

\begin{proof} {\it Step one}:
We want to prove that \eqref{Lambda r0 minimizer section} holds whenever $F\subset\Om$, $\diam(F\Delta E)<2r_0$
and $F\cap A^c=E\cap A^c$. Of course, without loss of generality, we may assume that
\[
I(F A,F^c \Om)+I(F A^c,F^c \Om  A)+\s\,I(F A,\Om^c)<\infty\,.
\]
Since $I(F A,\Om^c)\le I(\Om A,\Om^c)<\infty$, this implies
\begin{equation}
  \label{inparticular}
  I(F A,F^c \Om)+I(F A^c,F^c \Om  A)<\infty
\end{equation}
and hence
\[
\int_A\,w_F=I(AF,F^c)=I(AF,F^c\Om)+I(AF,\Om^c)\le I(AF,F^c\Om)+P(\Om)<\infty\,.
\]
Similarly,
\begin{eqnarray*}
\int_Aw_{F^c}&=&I(AF^c,F)=I(AF^c,AF)+I(AF^c,A^cF)
\\
&\le&
I(AF^c,AF)+I(AF^c\Om^c,A^cF)+I(AF^c\Om,A^cF)
\end{eqnarray*}
where $I(AF^c,AF)\le\int_Aw_F<\infty$, $I(AF^c\Om^c,A^cF)\le P(\Om)<\infty$, and $I(AF^c\Om,A^cF)<\infty$ by \eqref{inparticular}. We have thus proved that in showing \eqref{Lambda r0 minimizer section} for a given $F\subset\Om$ with $\diam(F\Delta E)<2r_0$
and $F\cap A^c=E\cap A^c$, we can directly assume that
\begin{equation}
  \label{wFwFcL1}
  w_F\,,w_{F^c}\in L^1(A)\,.
\end{equation}
Now we fix a bounded set $W$ with $w_W,w_{W^c}\in L^1(\R^n)$ such that
$F\Delta E\cc W\cc A$ and $\diam(W)<2r_0$ (we can achieve this by taking $W$ in the form of a finite union of balls, say).
 Our goal is thus proving that
\begin{eqnarray}\label{fine}
&&I(E  W,E^c \Om  W)
+I(E  W,E^c \Om  W^c)
+I(E  W^c,E^c \Om  W)
+\s\,I(E  W,\Om^c)
\\  &\le&\nonumber
  I(F  W,F^c \Om  W)+
I(F  W,E^c \Om  W^c)
+I(E  W^c,F^c \Om  W)
+\s\,I(F  W,\Om^c)
+\Lambda\,|E\Delta F|\,.
\end{eqnarray}
To this end we set $F_j=(F\cap\Om_j\cap W)\cup(E_j\cap W^c)$,
and test the minimality inequality of $E_j$ (see~\eqref{Lambda r0 minimizer}) on $F_j$. In this way we find
  \begin{eqnarray}
    \label{Lambda r0 minimizer j}
  &&I(E_j  W,E_j^c \Om_j  W)
+I(E_j  W,E_j^c \Om_j  W^c)
+I(E_j  W^c,E_j^c \Om_j  W)
+\s\,I(E_j  W,\Om_j^c)
  \\ \nonumber
  &\le&
  I(F \Om_j  W,F^c \Om_j  W)+
I(F \Om_j  W,E_j^c \Om_j  W^c)
+I(E_j  W^c,F^c \Om_j  W)
+\s\,I(F \Om_j  W,\Om_j^c)
\\\nonumber
&&+\Lambda\,|E_j\Delta F_j|.
  \end{eqnarray}
We claim that in the limit $j\to\infty$, \eqref{Lambda r0 minimizer j} implies \eqref{fine}.
By Fatou's lemma and \eqref{EjOmj convergono}, the inferior limit as $j\to\infty$ of the sum of first three terms on the left-hand side of \eqref{Lambda r0 minimizer j} is bounded from below by the corresponding sum on the left-hand side of \eqref{fine}. We thus have to address the behavior of the two $\s$-terms in \eqref{Lambda r0 minimizer j}, and of the first three terms appearing on its right-hand side.
We start with the first of these terms, and find by~\eqref{kernelclass 0} and~\eqref{DEF WOM} that
\begin{eqnarray*}
|I(F \Om_j  W,F^c \Om_j  W)
-I(F \Om  W,F^c \Om  W)|&\le&
\big|I(F \Om_j  W,F^c \Om_j  W)
-I(F \Om_j  W,F^c \Om  W)\big|
\\&&+
\big|I(F \Om_j  W,F^c \Om  W)
-I(F \Om  W,F^c \Om  W)\big|
\\
&\le&\int_{(\Om_j\Delta \Om)\cap W}w_{F^c}+\int_{(\Om_j\Delta \Om)\cap W}w_{F}\,.
\end{eqnarray*}
By \eqref{EjOmj convergono} and \eqref{wFwFcL1} we thus find
\begin{equation}\label{FMS:2}
\lim_{j\to+\infty}
I(F \Om_j  W,F^c \Om_j  W)=
I(F \Om  W,F^c \Om  W)\,.
\end{equation}
Now we claim that if~$G_j\subset\Om_j$ and~$G_j\to G$ in $L^1_{{\rm loc}}(A)$, then
\begin{equation}\label{8:AUS}
\lim_{j\to+\infty} I(G_j  W,\Om_j^c)=
I(G  W,\Om^c).
\end{equation}
To this end, we recall that \eqref{DLVPw} implies
\[
\lim_{j\to\infty}\int_{U_j}w_{\Om_j}=0\qquad\mbox{whenever}\qquad\lim_{j\to\infty}|U_j|=0\,,
\]
see \cite[Theorem 1.38]{AFP}. Since, by definition, \eqref{DLVPw} gives us $\int_{\R^n}u\,w_{\Om_j}\to\int_{\R^n}u\,w_\Om$ for every $u\in L^\infty_{{\loc}}(A)$, we conclude that
\begin{equation*}
\begin{split}
& |I(G_jW,\Om_j^c)-I(G_jW,\Om^c)|=
\left| \int_{\R^n} 1_{G_j\cap W}(x)\,\big( w_{\Om_j}(x)-
w_{\Om}(x)\big) \,dx
\right|\\
&\qquad\le
\left| \int_{\R^n} 1_{GW}(x)\,\big( w_{\Om_j}(x)-
w_{\Om}(x)\big) \,dx
\right|+
\int_{(G_j\Delta G) \cap W}\big(w_{\Om_j}(x)+
w_{\Om}(x)\big) \,dx
\;\longrightarrow 0
\end{split}
\end{equation*}
as~$j\to+\infty$. {F}rom this last fact and thanks to $w_\Om\in L^1(\R^n)$ we have
\begin{eqnarray*}
&& |I(G_jW,\Om_j^c)-I(GW,\Om^c)|\\
&\le& |I(G_jW,\Om_j^c)-I(G_jW,\Om^c)|+|I(G_jW,\Om^c)-I(GW,\Om^c)|
\\ &\le& |I(G_jW,\Om_j^c)-I(G_jW,\Om^c)|+\int_{(G_j\Delta G)W} w_\Om(x)\,dx\;\longrightarrow 0
\end{eqnarray*}
as~$j\to+\infty$, which proves~\eqref{8:AUS}. We now exploit \eqref{8:AUS}
with $G_j=E_j$ and with~$G_j=F\cap\Om_j$ to take care of the $\s$-terms in
\eqref{Lambda r0 minimizer j} and find
\begin{equation}\label{FMS:3}
\lim_{j\to+\infty} I(E_j  W,\Om_j^c)=I(E  W,\Om^c)
\qquad
\lim_{j\to+\infty} I(F\Om_j  W,\Om_j^c)=I(F\Om  W,\Om^c)\,.
\end{equation}
We are left to take care of the second and third terms on the right-hand side of \eqref{Lambda r0 minimizer j}.
To this end we notice that since $w_{W},w_{W^c}\in L^1(\R^n)$,
if~$G_j$, $L_j\subset\Om_j$ with $G_j\to G$ and $L_j\to L$ in $L^1_{{\rm loc}}(A)$, then
\begin{eqnarray*}
&& |I(G_jW,L_jW^c)-I(GW,LW^c)|\\
&\le& |I(G_jW,L_jW^c)-I(G_jW,LW^c)| +|I(G_jW,LW^c)-I(GW,LW^c)|\\
&\le& \int_{L_j\Delta L} w_{W^c}(x)\,dx
+ \int_{G_j\Delta G} w_{W}(x)\,dx
\;\longrightarrow 0
\end{eqnarray*}
as~$j\to+\infty$. By using this observation first
with~$G_j:=F\cap\Om_j$ and~$L_j:=E_j^c\Om_j$, and then with~$G_j:=F^c\Om_j$ and~$L_j:=E_j$,
we finally obtain that
\begin{equation}\label{0d6567}\begin{split}
&\lim_{j\to+\infty}
I(F \Om_j  W,E_j^c \Om_j  W^c)=
I(F \Om  W,E^c \Om  W^c)
\\
&\lim_{j\to+\infty}
I(E_j \Om_j  W^c,F^c \Om_j  W)
= I(E \Om_j  W^c,F^c \Om  W).\end{split}\end{equation}
We have thus completed the proof of \eqref{fine}.

\medskip

\noindent {\it Step two}: Let us now assume that $K=K_s^\e$ and that $\Om_j=\Om$ for an open set $\Om$ with $C^1$-boundary such that $\varrho(\eta,\Om)>0$ for every $\eta>0$, so that the density estimates of Theorem \ref{thm density estimate full} hold. Let us pick $x_j\in A\cap\ov{\Om\cap\pa E_j}$ and $x_j\to x$ for some $x\in A$, then $x\in \pa E$. By \eqref{volume density estimates}
\[
\frac1{C_0}\le \frac{|E_j\cap B_r(x_j)|}{r^n}\le 1-\frac1{C_0}\,,
\]
for every $r<\min\{\dist(x_j,\pa A),r_0,\,c_*\varrho(\eta_1,\Om),\,c_*\,\e\}$ where $C_0=C_0(n,s,\s,\Lambda)$ and $c_*=c_*(n,s)$. As $j\to\infty$ we find
\[
\frac1{C_0}\le \frac{|E\cap B_r(x)|}{r^n}\le 1-\frac1{C_0}\,,
\]
for every $r<\min\{\dist(x,\pa A),r_0,\,c_*\varrho(\eta_1,\Om),\,c_*\,\e\}$, that is $x\in\pa E$. Now let us consider $x\in\ov{\Om\cap\pa E}$, and assume that for some $\tau>0$ and for infinitely many values of $j$ we have $B_\tau(x)\cap\pa E_j=\varnothing$. Without loss of generality we may assume that either $|E_j\cap B_\tau(x)|=0$ or $|E_j^c\cap B_\tau(x)|=0$ for infinitely many $j$. In this way, by Fatou's lemma,
\[
0=\lim_{j\to\infty}I^\e_s(E_jB_\tau(x),E_j^cB_\tau(x))\ge I_s^\e(E B_\tau(x),E^cB_\tau(x))
\]
so that either $|E\cap B_\tau(x)|=0$ or $|E^c\cap B_\tau(x)|=0$, against the fact that $x\in\ov{\Om\cap\pa E}$ and thus the density estimates \eqref{volume density estimates} hold for $E$ at $x$, being $E$ a $(\Lambda,r_0,\s,K^\e_s)$-minimizer.
\end{proof}

As an application of Theorem \ref{DLVP} we obtain the following compactness result for blow-ups. For the sake of simplicity, we limit our analysis to the case $K=K_s$.

\begin{theorem}\label{BL:COMP:TH}
Let $n\ge2$, $s\in(0,1)$, $\s\in(-1,1)$, $\Lambda\ge0$, $r_0>0$, and let $A$ be an open set.
Let $E$ be a $(\Lambda,r_0,\s,K_s)$-minimizer in $(A,\Om)$ where $\Om$ is an open set with $C^1$-boundary in $A$ and with $\varrho_A(\eta,\Om)>0$ for every $\eta>0$, let $x_0\in A\cap\ov{\Om\cap\pa E}$, and
given a positive vanishing sequence $\{r_j\}_{j\in\N}$, set
\[
E_j=E^{x_0,r_j}=\frac{E-x_0}{r_j}\qquad \Om_j=\Om^{x_0,r_j}\,.
\]
Then there exists an half-space $H$ with $0\in\pa H$ and a set $E\subset H$ with $0\in\pa E$ such that, up to extracting a subsequence,
$E_j\to E$ and~$\Om_j\to\Om$ in $L^1_{\loc}(\R^n)$ as $j\to\infty$ and~$E$
is a $(0,\infty,\s,K_s)$-minimizer in~$(\R^n,H)$.
\end{theorem}

We shall need the following simple lemma.

\begin{lemma}\label{LEM WOM}
If~$f:\R^n\to\R^n$ iw a bi-Lipschitz diffeomorphism and $\Om\subset\R^n$, then
\[
w_{f(\Omega)}(x)\le C\, w_{\Omega}(f^{-1}(x))\qquad\forall x\in\R^n\,,
\]
with a constant~$C$ depending only on the Lipschitz constants of~$f$ and $f^{-1}$, and converging to $1$ when these Lipschitz constants converge to $1$.
\end{lemma}

\begin{proof}[Proof of Lemma \ref{LEM WOM}] Setting~$\tilde y=f^{-1}(y)$ and $\tilde x=f^{-1}(x)$ the area formula gives
$$ w_{f(\Omega)}(x)
= 1_{f(\Om)}(x)\,\int_{f(\Om^c)}\frac{dy}{|x-y|^{n+s}}
= 1_{\Om}(\tilde x)\,\int_{\Om^c}\frac{ Jf (\tilde y)\,d\tilde
y}{|
f(\tilde x)-f(\tilde y)|^{n+s}}\,.$$
We conclude as $\|Jf\|_{L^\infty(\R^n)}\le 1+C(n)\,|\Lip(f)-1|$ and~$|f(\tilde x)-f(\tilde y)|\ge \Lip(f^{-1})\,|\tilde x-\tilde y|$.
\end{proof}

\begin{proof}[Proof of Theorem \ref{BL:COMP:TH}]
  By regularity of $\pa\Om$ we have that $\Om_j\to H$ in $L^1_{{\rm loc}}(\R^n)$, while if we set $A_j=(A-x_0)/r_j$, then $A_j\to\R^n$ in $L^1_{{\rm loc}}(\R^n)$ as $x_0\in A$. By Remark \ref{remark blowups}, $E_j$ is a $(\Lambda\,r_j^s,r_0/r_j,\s,K_s)$-minimizer in $(A_j,\Om_j)$. Let us fix $\tau>0$ and $R>0$ and notice that for $j$ large enough we certainly have that $E_j\cap B_{2R}$ is a $(\tau,\tau^{-1},\s,K_s)$-minimizer in $(B_R,\Om_j\cap B_{2R})$. Let us also notice that by \eqref{remark stable blowup} we can certainly assume that
  \begin{eqnarray*}
  \varrho_{B_R}(\eta,\Om_j\cap B_{2R})&=&\varrho_{(B_{Rr_j}(x_0))^{x_0,r_j}}\Big(\eta,(\Om\cap B_{2\,R\,r_j}(x_0))^{x_0,r_j}\Big)
    \\
    &=&\frac1{r_j}\,\varrho_{B_{Rr_j}(x_0)}\Big(\eta,\Om\cap B_{2\,R\,r_j}(x_0)\Big)\ge\frac{\varrho_A(\Om,\eta)}{r_j}
  \end{eqnarray*}
  In particular, for $\eta_1=\eta_1(n,s,\s)$ as in Theorem \ref{thm density estimate full} we have
  \[
  \theta:=\inf_{j\in\N}\varrho_{B_R}(\eta,\Om_j\cap B_{2R})>0\,.
  \]
  Now we show that, up to extracting a subsequence, $E_j\to E$ in $L^1_{{\rm loc}}(B_{2R})$ for some set $E\subset H$. Indeed, by \eqref{upper perimeter estimate} we have
  \[
   I_s(E_j B_r(x),(E_j B_r(x))^c)\le C_0\,r^{n-s}\,,
  \]
  whenever $B_r(x)\subset B_R$ and $r<\min\{\tau^{-1},c_*\,\theta\}$ where $C_0=C_0(n,s,\s,\Lambda)$, $c_*=c_*(n,s)$ and $\theta>0$ is as above. By a covering argument we see that
  \[
  \sup_{j\in\N}I_s(E_jW,E_j^c)<\infty
  \]
  for every $W\cc B_{2R}$. In particular, $E_j\to E$ in $L^1_{{\rm loc}}(B_{2R})$ for some set $E$, and the fact that $E\subset H$ follows immediately from $E_j\subset\Om_j$ and $\Om_j\to H$ in $L^1_{{\rm loc}}(\R^n)$.

  We claim that $w_{\Om_j\cap B_{2R}}$ converges weakly in $L^1(B_R)$ to $w_{H\cap B_{2R}}$. Indeed, there exists a bi-Lipschitz family of diffeomorphisms $f_j:\R^n\to\R^n$ such that $f_j(0)=0$, $f_j(\Om_j\cap B_{2R})=H\cap B_{2R}$ and
  \[
  (1-\de_j)\,|x-y|\le |f_j(x)-f_j(y)|\le (1+\de_j)\, |x-y|\qquad\forall x,y\in\R^n\,
  \]
  where $\de_j\to0$. In particular, by Lemma \ref{LEM WOM},
  \[
  (1-C(n)\,\de_j)\,w_{H\cap B_{2R}}\le
  w_{\Om_j\cap B_{2R}}\le (1+C(n)\,\de_j)\,w_{H\cap B_{2R}}\qquad\mbox{on $\R^n$}\,.
  \]
  This proves our claim.

  Since $P_s(H\cap B_{2R})<\infty$ we can apply Theorem \ref{DLVP} to conclude that, for every $\tau>0$, $E$ is $(\tau,\tau^{-1},\s,K_s)$-minimizer in $(B_R,H\cap B_{2R})$. By the arbitrariness of $\tau$, $E$ is $(0,\infty,\s,K_s)$-minimizer in $(B_R,H\cap B_{2R})$. By the arbitrariness of $R$,
  $E$ is $(0,\infty,\s,K_s)$-minimizer in $(\R^n,H)$. Again by Theorem \ref{DLVP}, since $0\in\ov{\Om_j\cap\pa E_j}$ for every $j$, it follows that $0\in\pa E$.
\end{proof}

\bibliography{references}
\bibliographystyle{is-alpha}

\end{document}